\documentclass[12p, reqno]{amsart}
\usepackage{amsfonts}
\usepackage{amsmath,amsthm,graphicx,mathrsfs,url   }
\usepackage{amssymb}
\usepackage{enumitem}

\usepackage{graphicx}
\usepackage[usenames,dvipsnames]{color}
\usepackage[colorlinks=true, linkcolor=Red, citecolor=Green]{hyperref}
\usepackage[displaymath,mathlines]{lineno}

\def\squarebox#1{\hbox to #1{\hfill\vbox to #1{\vfill}}}

\newcommand{\R}{{\mathbb R}}
\newcommand{\C}{{\mathbb C}}
\newcommand{\N}{{\mathbb N}}

\renewcommand{\Re}{\mathop{\rm Re}\nolimits}
\renewcommand{\Im}{\mathop{\rm Im}\nolimits}

\theoremstyle{plain}

\newtheorem{thm}{Theorem}[section]

\newtheorem{lem}{Lemma}[section]
\newtheorem{rem}{Remark}
\newtheorem{prop}{Proposition}[section]

\renewcommand{\theprop}{\thesection.\arabic{prop}}

\renewcommand{\theequation}{\thesection.\arabic{equation}}
\renewcommand{\therem}{\thesection.\arabic{rem}}

\numberwithin{equation}{section}

\begin{document}

\title[eigenvalues and resonances of dissipative operator] {Eigenvalues and resonances of dissipative acoustic operator for strictly convex obstacles}

\author[V. Petkov]{Vesselin Petkov}


\address{Universit\'e de Bordeaux, Institut de Math\'ematiques de Bordeaux, 351, Cours de la Lib\'eration,
33405 Talence, France}
\email{petkov@math.u-bordeaux.fr} 

\def\re{\mbox{\rm Re}\:}
\def\im{\mbox{\rm Im}\:}
\def\R{{\mathbb R}}
\def\C{{\mathbb C}}
\def\S{{\mathbb S}}
\def\ts{\tilde{\sigma}}
\def\p{{\mathcal P}}
\def\hc{{\mathcal H}}
\def\pa{\partial}
\def\om{\omega}
\def\ii{{\bf i}}
\def\rc{{\mathcal R}}
\def\nc{{\mathcal N}}
\def\cc{{\mathcal C}}
\def\tr{\tilde{r}_{0}}
\def\oph{Op_{h}(\chi_{h})}
\def\fh{f^{-}_{\delta}}
\def\f3{\frac{3}{2}\delta}
\def\ep{\epsilon}
\def\tl{\tilde{\lambda}}
\def\hd{\la hD\ra}
\def\supp{{\rm supp}\:}
\def\th{\tilde{h}}
\def\lc{{\mathcal L}}
\def\pu{\partial_t u}
\def\la{\langle}
\def\ra{\rangle}
\def\12{\frac{1}{2}}
\def\tl{\tilde{\lambda}}
\def\lx{\la x \ra}
\def\phi{\varphi}
\def\epsilon{\varepsilon}
\def\kappa{\varkappa}
\def\ep{\epsilon}
\def\lr#1{\langle{#1}\rangle}
\def\dif{\partial}
\def\lam{\lambda}
\def\al{\alpha}
\def\be{\beta}
\def\gam{\gamma}
\def\R{{\mathbb R}}
\def\Ss{{\mathbb S}}
\def\dem{\noindent    {\bf Proof.} }
\def\esp{\vspace{8pt}}
\def\dc{{\mathcal D}}
\def\oc{{\mathcal O}}
\def\Op{{\rm Op}}

\maketitle
\begin{abstract}
We examine the wave equation in the exterior of a strictly convex bounded domain $K$ with dissipative boundary condition $\pa_{\nu} u - \gamma(x) \pa_t u = 0$ on the boundary $\Gamma$ and $0 < \gamma(x) <1, \:\forall x \in \Gamma.$ The solutions are described by a contraction semigroup $V(t) = e^{tG}, \: t \geq 0.$ The  poles $\lambda$ of the meromorphic incoming resolvent $(G - \lambda)^{-1}: \: \hc_{comp} \rightarrow \dc_{loc}$  are eigenvalues of G if $\Re \lambda < 0$ and incoming resonances if $\Re \lambda > 0$. We obtain sharper results for the location of the eigenvalues of $G$ and incoming resonances in $\Lambda = \{\lambda \in \C:\: |\Re \lambda| \leq C_2(1 + |\Im \lambda|)^{-2},\: |\Im \lambda| \geq A_2 > 1\}$ and we prove a Weyl formula for their asymptotic. For $K = \{x \in \R^3:\:|x| \leq 1\}$ and $\gamma$ constant we show that $G$ has no eigenvalues so the Weyl formula concerns only the incoming resonances.
\end{abstract}



\section{Introduction}

Let $K \subset \R^d,$ $d \geq 3$, $d$ odd, be a bounded non-empty domain with $C^{\infty}$ strictly convex boundary $\Gamma.$   Let $\Omega = \R^d \setminus \bar{K}$ be connected and $K \subset \{x \in \R^d:\:|x| < \rho_0\}.$ 
Consider the boundary problem
\begin{equation} \label{eq:1.1}
\begin{cases} u_{tt} - \Delta_x u = 0 \: {\rm in}\: \R_t^+ \times \Omega,\\
\partial_{\nu}u - \gamma(x) \pa_t u= 0 \: {\rm on} \: \R_t^+ \times \Gamma,\\
u(0, x) = f_1, \: u_t(0, x) = f_2 \end{cases}
\end{equation}
with initial data $(f_1, f_2) \in \hc = H^1(\Omega) \times L^2(\Omega).$
Here $\nu(x)$ is the unit outward normal at $x \in \Gamma$ pointing into $\Omega$ and $\gamma(x) > 0$ is a $C^{\infty}$ function on $\Gamma.$ The solution of the problem (\ref{eq:1.1}) has the form $V(t)f = e^{tG} f,\: t \geq 0$, where $V(t)$ is a contraction semi-group in ${\mathcal H}$ with  generator
$$ G = \Bigl(\begin{matrix} 0 & 1\\ \Delta & 0 \end{matrix} \Bigr).$$
The operator $G$  has domain  $\dc$ given by the closure in the graph norm
$$|\|f\| | = (\|f\|_{{\mathcal H}}^2 + \|G f\|^2_{{\mathcal H}})^{1/2} $$
 of functions $f = (f_1, f_2) \in C_{(0)}^{\infty} (\R^d) \times C_{(0)}^{\infty} (\R^d)$ satisfying the boundary condition $\partial_{\nu} f_1 - \gamma f_2 = 0$ on $\Gamma.$ It is well known  that the point spectrum $\sigma_p(G)$ of $G$ in $\C_{-} = \{z \in \C:\:\Re z < 0\}$ is formed by isolated eigenvalues with finite multiplicity, $\sigma_p(G) \cap \ii \R = \emptyset$ and $G$ has continuous spectrum $\sigma_c(G) = \ii \R$ (see Section 8, \cite{LP1}).
Notice that if $Gf =\lambda f$ with $0 \neq f \in \mathcal D, \: \Re \lambda < 0,$ we have 
\begin{equation} \label{eq:1.2}
\begin{cases} (\Delta - \lambda^2) f_1 = 0 \:{\rm in}\: \Omega,\\
\partial_{\nu} f_1 -  \lambda \gamma f_1 = 0\: {\rm on}\: \Gamma \end{cases}
\end{equation}
and $u(t, x) = V(t) f = e^{\lambda t} f(x) $ is a solution of (\ref{eq:1.1}) with exponentially decreasing global energy. Such solutions are called {\bf asymptotically disappearing} and their existence  is important for the inverse scattering problems. We refer to \cite{P1}, \cite{P2}, \cite{P3} for a description of the relation of asymptotically disappearing solutions  to scattering theory. On the other hand, a solution $u(t,x)$ of (\ref{eq:1.1}) is called {\bf disappearing} if there exists $T > 0$ such that $u(t, x) \equiv 0$ for $t \geq T.$ The existence of a such solution implies that the space
$$\hc(T) = \{u(t, x):\: u(t, x) \equiv 0\: {\rm for}\: t \geq T\}$$
is infinite dimensional. Majda \cite{Ma1} proved that in the case $0 \leq \gamma (x) < 1, \forall x \in \Gamma$ and $C^{\infty}$ boundary $\Gamma$ and in the case $\gamma(x) > 1, \forall x \in \Gamma$ and real analytic boundary $\Gamma$ there are no disappearing solutions.

 Let ${\mathscr S}(z): L^2({\mathbb S}^{d-1}) \rightarrow L^2({\mathbb S}^{d-1})$ be the scattering operator defined in Section 3, \cite{LP1} for $d$ odd.  The existence of $z_0, \im z_0 > 0,$ such that ${\rm Ker} {\mathscr S}(z_0) \neq \{0\}$ implies that $\ii z_0 \in \sigma_p(G)$ (see Theorem 5.6 in \cite{LP1}). Consequently, if $\sigma_p(G) = \emptyset,$ the operator ${\mathscr S} (z)$ has trivial kernel at all regular points of ${\mathscr S}(z)$ with $\Im z \geq 0.$ (The same property holds for the scattering operators related to Dirichlet and Neumann problems since ${\mathscr S}^{-1} (z) = {\mathscr S}^*(\bar{z})$ for these problems, provided that ${\mathscr S}^{-1}(z)$ exists). In  \cite{LP1} the  energy space $\hc$ was decomposed with three orthogonal parts
$$\hc = D_{-}^{a} \oplus K^{a} \oplus D_{+}^a, \: a > \rho_0$$
  and we have the relations 
    \begin{eqnarray} \label{eq:1.3}
 U_0(-t) V(t) g = V^*(t) U_0(t) g = g, \: g \in D_{+}^a,\nonumber\\
  U_0(t) V^*(t) g = V(t) U_0(-t) g = g,\: g \in D_{-}^a,
 \end{eqnarray}
$ U_0(t)$ being the unitary group related to Cauchy problem for the free wave equation  in $\R^d$ (see Section 2 in \cite{LP1} for the definition of $D_{\pm}^a, \: K^a$ and the above relations).  A function $f \in \hc$ is incoming (resp. outgoing) if its component in $D_{+}^a$ (resp. $D_{-}^a$) vanishes. It is easy to see that if $f$ is an eigenfunction of $G$ with eigenvalues $\lambda$ , then $f$ is incoming. Indeed, for every $g \in D_{+}^a$ we get
  \[ (f, g) = (f, V^*(t) U_0(t) g) = (e^{\lambda t} f, U_0(t) g) \rightarrow 0 \: {\rm as}\: t \to \infty\]
   and the solution $V(t) f$ remains incoming for all $t \geq 0$. Here $(\bullet , \bullet)$ denotes the scalar product in $\hc.$ This leads to difficulties in the inverse scattering problems.

The location in $\C_{-}$ of the eigenvalues of $G$ has been studied in \cite{P1} improving previous results of Majda \cite{Ma}. It was proved in \cite{P1} that if $K$ is the unit ball $B_3 =\{x \in \R^3: \;|x|\leq 1\}$ and $\gamma \equiv 1,$ the operator $G$ has no eigenvalues. For this reason we study the cases 
$$(A): \max_{x \in \Gamma}  \gamma(x)  < 1,\:(B): \min_{x \in \Gamma} \gamma(x) > 1.$$
 
The case (B) has been examined in \cite{P1} and it was proved that for every $0 < \ep \ll 1$ and every $N  \in \N, \: N \geq 1$ the eigenvalues lie in $\Lambda_{ \ep} \cup {\mathcal R}_{N}$, where 
$$\Lambda_{\ep} = \{z \in \C: \: |\Re z | \leq C_{\ep} (1 + |\Im z|^{1/2 + \ep} ), \: \Re z < 0\},$$
$${\mathcal R}_N = \{z \in \C:\: |\Im z | \leq A_N (1 + |\Re z|)^{-N} ,\: \Re z < 0\}.$$
Moreover, it was shown in \cite{P1} that for strictly convex obstacles $K$ there exists $R_0 > 0$ such that the eigenvalues lie in $\{z \in \C:\: |z| \leq R_0,\: \Re z < 0\} \cup {\mathcal R}_N.$ Next, for obstacles with arbitrary geometry a Weyl formula for the asymptotic of the eigenvalues lying in
$${\mathcal R} = \{ \lambda \in \C:\: |\Im \lambda| \leq C_1( 1 + |\Re \lambda|)^{-2}, \: \Re \lambda\leq -C_0 \leq -1\}$$ 
has been  established in \cite{P2}. This formula has the form
\begin{eqnarray*} 
\sharp \{ \lambda_j \in \sigma_p(G) \cap {\mathcal R}:\: |\lambda_j | \leq r, \: r \geq C_{\gamma}\} \nonumber \\
=  \frac{\omega_{d-1}}{(2 \pi)^{d -1}}\Bigl( \int_{\Gamma} (\gamma^2(x) - 1)^{(d-1)/2} dS_x\Bigr) r^{d-1} + {\mathcal O}_{\gamma} (r^{d-2}),\: r \to \infty,
\end{eqnarray*} 
$\omega_{d-1}$ being the volume of the unit ball $\{x \in \R^{d-1}: |x| \leq 1\}.$ A similar result for the Maxwell system with dissipative boundary conditions and $\gamma(x) \neq 1, \forall x \in \Gamma,$ has been proved in \cite{P3}. 

It is important to note that in the case (B) the properties of the exterior Dirichlet-to Neumann operator only in the elliptic region ${\mathscr E}$ are important (see Section 3 for the definition the hyperbolic, glancing and elliptic regions ${\mathscr H}, \: {\mathscr G}, \: {\mathscr E},$ respectively). The analysis in \cite{P2} was similar to  that in \cite{SjV}, where the operator on the boundary is not invertible in some manifold included in the elliptic region. This phenomenon creates Rayleigh resonances for elasticity operator.

The study of the case (A) is more difficult and only some results for the location of eigenvalues are known. In \cite{P1} the previous result of Majda \cite{Ma} has been improved and it was shown that the eigenvalues  of $G$ for every $0 < \ep \ll 1$ are included in the region $\Lambda_{\ep}.$
By using the results in \cite{V1}, it is possible to improve the eigenvalue free regions replacing $\Lambda_{\ep}$ by $\{z \in \C:\: -A_0 \leq \Re z < 0\}$ with sufficiently large $A_0 > 0.$ The same eigenvalues free region for strictly convex obstacles $K$ and $d = 3$ can be obtained from Theorems 2.1 and 2.2 in \cite{I}.

The eigenvalues in the half plane $\C_{-}$ are the poles of the meromorphic incoming resolvent $(G - \lambda)^{-1}: \hc \rightarrow \dc$  (see Section 2 for the definition of incoming/outgoing solutions). Since $G$ has continuous spectrum on $\ii \R,$ it is convenient to extend the incoming resolvent. Let  $R_D(\lambda) = (-\Delta_D  + \lambda^2)^{-1}: L^2(\Omega) \rightarrow {\bf D}$ be the incoming resolvent of the Dirichlet Laplacian with domain ${\bf D}$ which is analytic in $\C_{-}$. The resolvent $R_D(\lambda) : L^2(\Omega)_{comp} \rightarrow {\bf D} _{loc}$ has a meromorphic continuation in $\C$ and the same is true for the incoming resolvent related to Neumann problem. These properties make possible to extend the incoming resolvent $(G - \lambda)^{-1} : \hc_{comp} \rightarrow \dc_{loc}$ from $\C_{-}$ to $\C$  as meromorphic function (see Section 2). The poles of this continuation in $\C_{+} = \{ z \in \C: \:\Re z > 0\}$ are called   {\bf incoming resonances} and noted by ${\rm Res}\:(G)$. 

 The outgoing resolvent
 \begin{equation} \label{eq:1.4}
 (G - \lambda)^{-1} = -\int_0^{\infty} e^{t G} e^{-\lambda t} dt,\:\Re \lambda > 0
 \end{equation} 
is analytic for $\Re\lambda > 0$. 
We obtain the {\it outgoing resonances} as the poles of the meromorphic continuation of the outgoing resolvent $(G - \lambda)^{-1}:\: \hc_{comp} \rightarrow \dc_{loc}$ for $\lambda \in \C$. These outgoing resonances  have been studied in Section 5, \cite{LP1}. Notice that $z \in \C_{-}$ is outgoing resonance if and only if $-\ii z$ is a pole of the scattering operator ${\mathscr S}(z)$ (see Lemma 5.2 and Lemma 5.4 in \cite{LP1}). 

The examination of incoming resonances is much more difficult. To our best knowledge, it seems that in the literature there are no results concerning their location and existence and for first time these problems are studied in our paper.  Obviously,
   $$((G - \lambda)^{-1})^* = (G^* - \bar{\lambda})^{-1} = -(- G^* - (- \bar{\lambda}))^{-1},$$
where $G^*= \Bigl(\begin{matrix} 0 & - 1\\ -\Delta & 0 \end{matrix} \Bigr)$ is the adjoint operator with domain $\dc^*$ determined by the boundary condition $\pa_{\nu}f_1 +\gamma(x) f_2 = 0$ on $\Gamma.$ For $\gamma(x) \neq 0$ we obtain $G \neq - G^*$ and  if $w$ is an  incoming resonance with resonance state $f$, then $- \bar{w}$ is not an outgoing resonance with resonance state $\bar{f}.$ Formally, a representation of the incoming resolvent similar to (\ref{eq:1.4}) should be given by an integral from $- \infty$ to 0 involving the solution of the problem (\ref{eq:1.1}) in $\R^{-}_t \times \Omega.$
 Changing the time $ t = -s,$ the problem in $\R^{-}_t \times \Omega$ is transformed in
\begin{equation} \label{eq:1.5}
\begin{cases} u_{tt} - \Delta_x u = 0 \: {\rm in}\: \R_t^+ \times \Omega,\\
\partial_{\nu}u + \gamma(x) \pa_t u= 0 \: {\rm on} \: \R_t^+ \times \Gamma,\\
u(0, x) = f_1, \: u_t(0, x) = f_2. \end{cases}
\end{equation}
This problem for $0 < \gamma(x)$ is not $L^2$ well posed (see for instance, Theorem 1 and Corollary 1 in \cite{Ma}), while  for $0 < \gamma(x) < 1$ the problem (\ref{eq:1.5}) is  $C^{\infty}$ well posed but with loss of regularity (see \cite{I}, \cite{I1}). This leads to difficulties if we wish to obtain an integral representation of the incoming resolvent, to develop a scattering theory related to  (\ref{eq:1.5}) and find a relation of the incoming resonances with the poles of a suitably defined  scattering operator.

Let $c_0 = \min_{x \in \Gamma}  \gamma(x), \: c_1 = \max_{x \in \Gamma} \gamma(x).$ 
 In this paper we obtain sharper results for the location of eigenvalues and incoming resonances for strictly convex obstacles.
 Our  first result is the following
 \begin{thm} Let $K$ be strictly convex obstacle and let $0 < \gamma(x) < 1,\; \forall x \in \Gamma$. There exists $R_0 > 0$ and $A_2 \gg 1$ depending on $c_0$ and $c_1$ such that for every $N \in \N, \: N \geq 1,$ the eigenvalues of $G$ are located in the region 
 $$\Bigl(\{z \in \C:\: |z| \leq R_0\} \cup {\mathcal Q}_N\Bigr) \cap \{\Re z < 0\},$$
  where
$${\mathcal Q}_N = \{z \in \C:\: | \Re z |\leq B_N (1 + |\Im z|)^{-N}, \: |\Im z| \geq  A_2\}.$$
Moreover, for every $c > 0$ there exists  $D_c > 0$ such that the  incoming resonances lying in 
$$M_c = \{z \in \C:\: 0 < \Re z \leq c \log|\im \lambda| , \: |\im \lambda| \geq D_c\}$$ 
are located in 
$$\Bigl(\{z \in \C:\: |z| \leq R_0\} \cup {\mathcal Q}_N\Bigr) \cap \{\Re z > 0\}.$$
 \end{thm} 
 
 The constants $R_0, A_2$ are very large. In particular, in Section 4 we show that we must take in Theorem 1.1 
 $$A_2 \geq   \frac{1}{\min\big\{ c_0^2, \frac{\sqrt{1 - c_1^2}}{2}\big\}},$$
 hence  $A_2\nearrow + \infty$ when either $c_0 \searrow 0$ or $c_1 \nearrow 1$. These properties are natural since for Neumann problem ($\gamma \equiv 0$) we have no eigenvalues of $G$ and incoming resonances are located outside $M_c$, while for the ball $B_3 = \{x\in \R^3,\: |x| \leq 1\}$ and $\gamma \equiv 1$ there are no eigenvalues. The above theorem is similar to Theorem  2.1 in \cite{P1} concerning the case $\gamma(x) > 1,\: \forall x \in \Gamma$.  On the other hand, for the ball $B_3$ and $0 < \gamma < 1, \: \gamma \equiv {\rm const}$ in Appendix we show  that $\sigma_p(G) = \emptyset.$ We conjecture that the same is true for $0 < \gamma(x) < 1,\: \forall x \in \Gamma$ and strictly convex obstacles $K$ and hope to study this question in a further work. The statement of Theorem 1.1 concerning incoming resonances is similar to Theorem 1.1, (a) in \cite{StV} concerning outgoing resonances of the elasticity system for strictly convex obstacles.

To study the distribution of eigenvalues and incoming resonances we generalise the result in \cite{P1} (see Proposition 2.2) and prove a trace formula involving extended incoming resolvent $(G - \lambda)^{-1}: \hc_{comp} \rightarrow \dc_{loc}$ and integration over curves intersecting the imaginary axis.  Define the multiplicity of an eigenvalue or incoming resonance $\lambda$ by 
$${\rm mult}(\lambda) = {\rm tr}\frac{1}{2 \pi \ii}  \int_{|z- \lambda| = \ep}  (\lambda - G)^{-1}  dz,$$
where $0 < \ep \ll 1$ is sufficiently small and introduce the set
 $$\Lambda: = \{z \in \C: \: |\re z|  \leq C_2(1 + |\im z|)^{-2},\: \: |\im z| \geq A_2 > 1\}.$$ 
Our second result is a Weyl asymptotic for the eigenvalues and incoming resonances in $\Lambda.$

\begin{thm} Let $K$ be strictly convex and let $0 < \gamma(x) < 1,\: \forall x\in \Gamma.$ Then the counting function of the eigenvalues and incoming resonances of  $G$ lying in $\Lambda$ and taken with their multiplicities has the asymptotic
\begin{eqnarray} \label{eq:1.6}
\sharp \{ \lambda_j \in (\sigma_p(G) \cup {\rm Res}\:(G)) \cap \Lambda,\: |\lambda_j| \leq r, \: r \geq C_{\gamma}\} \nonumber \\
=  \frac{2\omega_{d-1}}{(2 \pi)^{d -1}}\Bigl( \int_{\Gamma} (1 - \gamma^2(x))^{(d-1)/2} dS_x\Bigr) r^{d-1} + {\mathcal O}_{\gamma} (r^{d-2}),\: r \to \infty,
\end{eqnarray}
$\omega_{d-1}$ being the volume of the unit ball $\{x \in \R^{d-1}:\: |x| \leq 1\}.$
\end{thm} 
The constant $C_{\gamma}$ depends on $\gamma$ and $C_{\gamma} \nearrow +\infty$ when either $c_0 \searrow 0$ or $c_1 \nearrow 1.$
The eigenvalues and incoming resonances are symmetric with respect to real axis and this explains the factor 2 in (\ref{eq:1.6}).
If the conjecture for the absence of eigenvalues of $G$ is true, the asymptotic (\ref{eq:1.6}) will concern only the incoming resonances. In this direction our result will be similar to that in \cite{SjV} dealing with Weyl asymptotic of Rayleigh resonances. In particular, for the ball $B_3$ and $\gamma$ constant (\ref{eq:1.6}) implies the existence of incoming resonances lying in a small neighborhood of the imaginary axis. The location of the outgoing resonances is different and it seems possible following the approach in Section 4, \cite{DZ},  to show that for non trapping obstacles the outgoing resonances  in $\C_{-}$ are in a region bounded  by logarithmic curves. 
The existence of incoming resonances $z_j$ with $\Re z_j \searrow 0$ is related to the loss of exponential decay of the solutions of (\ref{eq:1.5}). Ikawa \cite{I1} studied the exponential decay  for the problem (\ref{eq:1.5}) with boundary condition
$$\partial_{\nu}u + \gamma(x) \pa_t u + d(x) u= 0 \: {\rm on} \: \R_t^+ \times \Gamma,$$
assuming $\gamma(x) < 1$ and $d(x) \leq - d_0 < 0, \: \forall x \in \Gamma$ with large $d_0 > 0.$ In our case we have $d(x) \equiv 0$ and the existence of incoming resonances close to $\ii \R$  for $B_ 3$ should imply that such decay of solutions is not possible. We don't study this problem in this paper.

Our approach is based on the following ideas. The behaviour of the exterior Dirichlet-to-Neumann operator $N(\lambda), \: \lambda \in \Lambda$ for $\Delta - \lambda^2$ in the hyperbolic region ${\mathscr H}$ is crucial for our analysis (see section 2 for the definition of $N(\lambda)$). Setting $\lambda =  \frac{\ii \sqrt{z}}{h}, \: 0 < h \ll 1,$ we are going  to study the Dirichlet problem (\ref{eq:2.12}) and the corresponding semi-classical Dirichlet-to-Neumann operator $\nc(z, h)$ for $z = 1 + \ii \im z$ and  $- c h |\log h| \leq \im z \leq h^{\ep},\: c > 0$ (see Section 2). A semi-classical parametrix for $\nc(z, h)$  and $|\im z| \leq h^{2/3}$ has been constructed in \cite{Sj}. We need a more precise information for the parametrix in $\mathscr H$ and for this purpose we present in subsection 3.1 a construction similar to that in Section 4, \cite{V3}. The trace formula (\ref{eq:2.7}) can be transformed into (\ref{eq:2.9}) with integration with respect to complex semi-classical parameter $\th = \frac{\ii}{\lambda}, \: 0 < \Re \th \ll 1.$ For the analysis in Sections 5 and 6  it is necessary to have a parametrix $\tilde{T}_N(\th)$ for the operator $\tilde{\nc}(\th)$ related to the problem (\ref{eq:3.14}) in ${\mathscr H}$ which is holomorphic in the domain
 $$ L= \{\th \in \C: \: |\im \th| \leq C_0 |\th|^2, \: 0 < \re \th \leq h_0\}.$$
  To do this, we make another construction in subsection 3.2 following that in Appendix A.2 in \cite{StV}.
     
     The location of eigenvalues and incoming resonances in $\Lambda$ is studied in Section 4. The idea is to prove that the eigenfunction or resonances states are supported in the hyperbolic region ${\mathscr H}.$  An eigenfunction/incoming resonance state  $f$ of $G$ satisfies the equation 
     \begin{equation}\label{eq:1.7}
     (\nc(z, h) - \sqrt{z}  \gamma(x))f = 0.
     \end{equation}
      We use a microlocal partition of unity $f = Q_{\delta}^{-} f + Q_{\delta}^{0} f + Q_{\delta}^{+} f$ with $h$-pseudo-differential operators $Q_{\delta}^{+}, \: Q_{\delta}^{-},\:Q_{\delta}^{0}$ having symbols supported in the regions ${\mathscr E}, \: {\mathscr H}, \: {\mathscr G}$, respectively. Here $0 < \delta \ll c_0^2$ depends of $c_0.$ By using (\ref{eq:1.7}) and the parametrix for $\nc(z, h) Q_{\delta}^{+}$ and $\nc(z, h) Q_{\delta}^{(0)}$ constructed in \cite{V1}, {\cite{V3}, \cite{V4}, \cite{Sj}, we show in subsections 4.1 and 4.2 (see Propositions 4.1, 4.2) that for small $\delta$ and $h$ we have $Q_{\delta}^{+} f = Q_{\delta}^{0} f = 0.$  This makes possible to reduce (\ref{eq:1.7}) to  
      $$(\nc(z, h) - \sqrt{z}  \gamma(x))Q_{\delta}^{-} f = 0.$$
      The proof is technically complicated because we cannot separate the terms involving $Q_{\delta}^{-}f$ and $Q_{\delta}^{(0)} f.$ 
      The last equation implies  $\im ((\nc(1, h) - \gamma) Q_{\delta}^{-}f, Q_{\delta}^{-} f) = 0$ because $\nc(1, h)$ is self-adjoint. We exploit a Taylor expansion 
       with respect to $\im z$ and one applies the approximation of  $\nc(z, h)Q_{\delta}^{-} $ by the parametrix $T_N(z, h)Q_{\delta}^{-}$ constructed in subsection 3.1. Our argument is similar to those used in \cite{V1}, \cite{P1} for the eigenvalues free regions created by the characteristic set of $N(z, h) - \sqrt{z} \gamma(x)$ in  ${\mathscr E}.$
      
         To obtain a Weyl formula, we are inspired by the strategy in \cite{SjV} (see also \cite{P2}, \cite{P3}). The crucial point is to construct an operator $P(\th)$ which is holomorphic in  $ L$ and self-adjoint for $\Re \th = h,\: 0 < h \leq h_0.$ For this purpose in Section 5 we modify outside ${\mathscr H}$ the parametrix $-\tilde{T}_N(\th)$ constructed in subsection 3.2 and consider 
 \[P(\th) : = -\tilde{T}_N(\th) \tilde{A}_0(\th) + \gamma(x) +(Id-  \tilde{A}_0(\th))  {\rm Op}\:\Bigl(\sqrt{1 + \th^2 r_0}\Bigr).\]  
  The operator $\tilde{A}_0(\th)$   is a suitable holomorphic extension of 
  $$A(h) = {\rm Op}(1 - \beta(h^2 r_0(x', \xi'))$$
   for complex $\th.$      
   Here $r_0(x', \xi')$ is the principal symbol of the Laplace-Beltrami operator $- \Delta\vert_{\Gamma}$ induced on $\Gamma$ and the function $\beta(t)\in C^{\infty}(\R:[0, 1])$ defined in subsection 4.3 has the properties:  $\beta(t) = 0$ for $t \leq 1- \delta,\: \beta(t) = 1$ for $t \geq 1- \delta/2,\: \beta'(t) \geq 0,\: \forall t \in \R$, with $\delta = \frac{c_0^2}{2}.$           
     The self-adjoint operator $P(h),\: 0 < h \leq h_0,$ has classical principal symbol 
  $$p_1 = -\sqrt{1 - h^2 r_0}(1-  \beta(h^2 r_0) + \gamma(x) + \beta(h^2 r_0) \sqrt{1 + h^2 r_0}.$$     
  Our modification has the following proprieties: (a) the glancing and elliptic region are cut-off, (b) in $\mathscr H$ we preserve the contribution from the parametric and this is important for the leading term in (\ref{eq:1.6}), (c) we have 
  $$\liminf_{(x', \xi') \to \infty} {\rm dist}\:(p_1(x', \xi') , ]-\infty, 0]) > 1.$$    Moreover,  the set
        \begin{equation} \label{eq:1.8}
\mathcal E = \{(x', \xi'):\: p_1(x', \xi') \leq 0\}  \iff \{ (x', \xi'): \:  r_0(x', \xi') \leq 1 - \gamma^2(x')\} \subset {\mathscr H}
\end{equation}   
is  independent of the choice of $\delta$ and $\beta$ and the same is true for the characteristic set $\Sigma : = \{(x', \xi'):\: p_1(x', \xi') = 0\}$ (see (\ref{eq:5.7})). Consequently, the asymptotic of the counting function for the negative eigenvalues of $P(h)$ can be obtained from the results of Section 10 in \cite{DS}. 

  The second point of the strategy is to study the eigenvalues of $P(h).$ As a preparation in Proposition 5.1, we establish the positivity of the operator $h \frac{d P(h)}{d h},$ choosing $h = o((1 - c_1^2))$.  Next our argument is similar to that applied in \cite{SjV}, \cite{P2}, \cite{P3}.  We introduce the eigenvalues 
$$\mu_1(h) \leq \mu_2(h) \leq ...\leq \mu_k(h) \leq ...$$
of $P(h)$ for $0 < h \leq h_0.$ By using Proposition 5.1 we prove that for every eigenvalue $\mu_k(h), \: k \geq k_0,$ which takes negative values for some $0 < h'_k \leq h_0$ there exists unique point $h'_k < h_k \leq h_0$ such that $\mu_k(h_k) = 0.$ For these values of $h$ the operator $P(h)$ is not invertible and it is necessary to avoid them working in some intervals not containing $h_k$. In Section 6 a trace formula for $P(\th)$ is given in Proposition 6.2. The final step is to compare the trace formula for $P(\th)$ with that in (\ref{eq:2.9}). By using the special choice of $P(\th)$, we show that these trace formulas differ by negligible term and we obtain a bijection between $]0, h_0] \ni h_k$ and eigenvalues/incoming resonances $\lambda_j \in \Lambda.$  Thus the problem is reduced to count for $0 < r^{-1}  \leq h_0$ the number of the negative eigenvalues $\mu_k(r^{-1})$ of $P(r^{-1}).$
 As in \cite{SjV}, \cite{P2}, this number is given by well known formula (\ref{eq:6.12}). Taking into account (\ref{eq:1.8}),  and we obtain the asymptotic (\ref{eq:1.6}).


\subsection*{Acknowledgements.} We would like to thank Johannes Sj\"ostrand for the useful and stimulating discussions concerning the construction of the parametrix. We thank also Jean-Fran\c cois Bony and Plamen Stefanov for the interesting discussions concerning the absence of eigenvalues.

\section{Preliminaries}

In this section we collect some facts from \cite{P1}, \cite{P2} and we prove a trace formula involving the extended incoming resolvent $(G - \lambda)^{-1}: \; \hc_{comp} \rightarrow \dc_{loc}.$  The results in this section hold for non trapping obstacles. For $\lambda \in \C$ introduce the exterior Dirichlet-to-Neumann map
$$N(\lambda): H^s(\Gamma) \ni f \longrightarrow \pa_{\nu} u\vert_{\Gamma} \in H^{s-1}(\Gamma),$$
where $u = K(\lambda) f$ is the solution of the problem
\begin{equation} \label{eq:2.1}
\begin{cases} (-\Delta +\lambda^2) K(\lambda)f = 0 \: {\rm in}\: \Omega,\\
K(\lambda)f = f \:{\rm on}\: \Gamma,\\
K(\lambda) f : \lambda-{\rm incoming}.
 \end{cases}
\end{equation}

A function $u(x)$  is $\lambda$-{\it outgoing} ($\lambda$-{\it incoming}) if there exists $R > \rho_0$ and $g \in L^2_{comp}(\R^d)$ such that
$$u(x) = R_0^{\pm}(\lambda) g,\: |x| \geq R,$$
where $R_0^{\pm}(\lambda) =(-\Delta_0 +\lambda^2)^{-1}$ is the outgoing (+) (incoming (-)) resolvent of the free Laplacian $- \Delta_0$ in $\R^d$. The resolvents $R_0^{\pm}(\lambda)$  are analytic in $\C$ for $d$ odd and they have kernels
 \begin{equation*}   R_0^{\pm} (\lambda, x, y) = \frac{(-1)^{(d-1)/2}}{2(2 \pi)^{(d-1)/2}} \Bigl(\frac{1}{r}\pa_r\Bigr)^{(d- 3)/2} \Bigl(\frac{ e^{\mp \lambda r}}{r}\Bigr)\Big\vert_{r =  |x-y|}.
\end{equation*} 

\begin{rem}
 Our definition of outgoing/incoming solutions coincides with that in Chapter IV, \cite{LP}. Setting $\lambda = \ii \mu,$ we see that the incoming resolvent $R_0^{-}(\ii \mu)= (- \Delta_0 - \mu^2)^{-1} $ for $\Im \mu >0$  is bounded form $L^2(\R^d)$ to $L^2(\R^d)$, hence $R_0^{-}(\ii \mu)$ becomes the outgoing resolvent defined in Section $3.1$, \cite{DZ}. Throughout our exposition the incoming resolvents are outgoing ones in the sense of \cite{G}, \cite{StV}, \cite{DZ}.
 \end{rem}

It is clear that if $G(f_1, f_2) = \lambda  (f_1, f_2)$ with $ \re \lambda < 0$, then $f_1 \in H^2(\Omega)$ is $\lambda$-incoming solution of $(-\Delta + \lambda^2) f_1 = 0$.
In particular, the incoming eigenvectors defined in Section 5,  \cite{LP1} are the eigenfunctions of $G$. 

Let $R_D (\lambda) = (-\Delta_D +\lambda^2)^{-1}$ be the incoming resolvent of the Dirichlet Laplacian $\Delta_D$ in $\Omega$ with domain ${\bf D} = H^2(\Omega) \cap H^1_0(\Omega)$
 which is analytic for $\lambda \in \C_{-}$. Let
 $${\bf D}_{loc} = \{ u \in L^2_{loc}(\Omega):\: \chi(x) \in C_0^{\infty}(\R^d), \: \chi(x) \equiv 1\: {\text in\:a\:neighborhood\:of} \bar{K}\: \Rightarrow \chi u \in {\bf D} \}. $$
 For $d$ odd the incoming resolvent has meromorphic extension 
 $$R_D (\lambda):\: L^2_{comp}(\Omega) \rightarrow {\bf D}_{loc}$$
  from $\C_{-}$ to $\C$ . The solution of the problem
(\ref{eq:2.1}) with $f \in H^{3/2}(\Gamma)$ has the form
\begin{equation} \label{eq:2.2}
u = e(f) + R_D(\lambda)((\Delta -\lambda^2) (e(f)),
\end{equation} 
where $e(f): H^{3/2}(\Gamma) \ni f \to e(f) \in H^{2}_{comp}(\Omega)$ is an extension operator. Clearly, we may find $\pa_{\nu}u\vert_{\Gamma}$ by applying (\ref{eq:2.2}).

 For non trapping obstacles and $d$  odd $R_{D}(\lambda)$ is analytic in
$$ \Lambda_0 = \{\lambda \in \C:\:\re \lambda < C_0 \log |\im \lambda|, \: |\im \lambda| \geq C_1 > 1\},\: C_0 > 0$$
and from (\ref{eq:2.2}) we deduce that operator $N(\lambda)$ is analytic in the same domain. Moreover, for non trapping obstacles the Laplacian with Neumann boundary conditions has no resonances in $\Lambda_0$ with suitable positive constants $C_0, C_1$ and this implies the analyticity of $N(\lambda)^{-1}.$\\

Going back to the problem (\ref{eq:1.2}), we write for $\lambda \in \Lambda_0$ the boundary condition  as 
\begin{equation} \label{eq:2.3}
\cc(\lambda)v: = (N(\lambda) - \lambda \gamma) v  = N(\lambda) \Bigl( Id -  N(\lambda)^{-1} \lambda \gamma \Bigr) v  = 0,\:v=  f_1\vert_{\Gamma} \in H^{3/2}(\Gamma).
\end{equation} 
For $ \lambda \in \Lambda_0$ the operator $\cc(\lambda):\:H^{s}(\Gamma) \rightarrow H^{s-1}(\Gamma)$ has the same singularities as $N(\lambda)$.
On the other hand, $N(\lambda)^{-1}:\: H^{s}(\Gamma) \rightarrow H^{s+ 1}(\Gamma)$ is compact operator and $\cc(\lambda_0)$ is invertible for some $\lambda_0$. Applying analytic Fredholm theorem (see \cite{P1}, \cite{P2}),  the operator $\cc(\lambda)$ is a meromorphic operator valued function for $\lambda \in \Lambda_0.$ Here and below a meromorphic operator valued function $B(z)$ means that $B(z)$  have Laurent expansion with finite number negative powers of $z$ and coefficients having finite rank.
 
 For the resolvent $(\lambda- G)^{-1}: {\mathcal H} \rightarrow {\mathcal D},\: \lambda \notin \sigma_p(G),\: \re \lambda < 0$ we proved in \cite{P1} the following trace formula.
 
\begin{prop} Let $\omega \subset \{ \lambda \in \C: \: \re 
\lambda < 0\}$ be a closed positively oriented curve without self intersections. Assume that  $\cc(\lambda)^{-1} $ has no poles on $\omega$ . Then
\begin{equation} \label{eq:2.4}
 {\rm tr}_{{\mathcal H}} \: \frac{1}{2 \pi \ii} \int_{\omega} (\lambda - G)^{-1} d\lambda = {\rm tr}_{H^{1/2}(\Gamma)} \:\frac{1}{2 \pi \ii} \int_{\omega} \cc(\lambda)^{-1}  \frac{\pa \cc}{\pa \lambda}(\lambda) d \lambda.
\end{equation}
\end{prop}
The left hand term in (\ref{eq:2.4}) is equal to the number of the eigenvalues of $G$ in the domain ${\mathcal U}$ bounded by $\omega$ counted with their multiplicities.
It is clear that $\lambda \in \C_{-}$ is an eigenvalues of $G$ if there exists a function $0 \neq f \in H^{3/2}(\Gamma)$ such that
\begin{equation} \label{eq:2.5}
\cc(\lambda) = N(\lambda) f - \lambda \gamma(x) f = 0.
\end{equation} 
Since $G$ has continuous spectrum on $\ii \R$ (see Section 8, \cite {LP1}), we cannot extend the trace formula (\ref{eq:2.4}) with the resolvent $(\lambda - G)^{-1}$ for curves intersecting the imaginary axis. To do this, we extend below the incoming resolvent.

  For $d$ odd,  $R_{D}(\lambda):\:L^2_{comp}(\Omega) \rightarrow {\bf D}_{loc},\: N(\lambda)^{-1}$ and $\cc(\lambda)$ are meromorphic operator valued functions in $\C$.  By using (\ref{eq:2.3}) and meromorphic Fredholm theorem, we conclude  that $\cc(\lambda)^{-1} $ will be meromorphic in $\C$. If $R_D(\lambda)$ and $N(\lambda)^{-1}$ are analytic in some domain ${\mathcal V} \subset \C$, we can apply the analytic Fredholm theorem to obtain that $\cc(\lambda)^{-1}$ is meromorphic for $\lambda \in {\mathcal V}$. In particular, for non trapping obstacles there exists $ a > 0$ such that $R_D(\lambda):\: L^2_{comp}(\Omega) \rightarrow {\bf D}_{loc}$ and $N(\lambda)^{-1}$ are analytic for $\re \lambda < a$ and for strictly convex obstacles $K$ this statement holds. It is easy to see that we can extend the incoming   resolvent $(G - \lambda)^{-1},\: \lambda \notin \sigma_p(G), \: \re \lambda < 0$ as meromorphic function $(G - \lambda)^{-1}:\: {\mathcal H}_{comp} \rightarrow {\mathcal D}_{loc}$ for $\lambda \in \C,$  
  where
  $${\mathcal D}_{loc} = \{ u \in {\mathcal H}_{loc},\: \chi(x) \in C_0^{\infty}(\R^d), \: \chi(x) \equiv 1 \: {\text in\: a\: neighborhood\:of} \bar{K}\: \Rightarrow \chi u \in {\mathcal D} \}.$$

Let $ \begin{pmatrix} u \\ w \end{pmatrix} = (G- \lambda)^{-1} \begin{pmatrix} f \\ g \end{pmatrix} $  with $(f, g) \in {\mathcal H}_{comp}$.
 For $\lambda \in \C_{-}$ we have $w = \lambda u + f,$
$$u = - R_D(\lambda)  (g + \lambda  f) + K(\lambda) q$$
with $q = u \vert_{\Gamma} $ and $K(\lambda)q$ determined by (\ref{eq:2.1}).
From the boundary condition
$$\pa_{\nu} \Bigl[ - R_D(\lambda) ( g + \lambda f) + K(\lambda) q\Bigr]\Big \vert_{\Gamma} -\lambda \gamma q- \gamma f\big\vert_{\Gamma}  = 0,$$
one gets 
$$q  = \cc(\lambda)^{-1} \Bigl[ \pa_{\nu}\Bigl (R_D^{+}(\lambda)  ( g + \lambda f)\Bigr)\Big\vert_{\Gamma}  + \gamma f\vert_{\Gamma} \Bigr],$$
provided that $\cc(\lambda)^{-1}$ exists. 
Thus we obtain 
\begin{equation} \label{eq:2.6}
 u = -  R_D(\lambda)   (g + \lambda f) +  K(\lambda) \cc(\lambda)^{-1} \Bigl[ \pa_{\nu}\Bigl (R_D(\lambda) ( g + \lambda f)\Bigr)\Big\vert_{\Gamma}  + \gamma f\vert_{\Gamma} \Bigr].
\end{equation}
The operator  $\cc(\lambda)^{-1} $ has meromorphic extension and 
$$R_D(\lambda) :\:  L^2_{comp}(\Omega)  \rightarrow {\bf D}_{loc},\: K(\lambda) :\: H^{3/2}(\Gamma) \rightarrow {\bf D}_{loc}$$
 are meromorphic operator valued functions.
 Consequently, $( u,  w) \in {\mathcal D}_{loc}$ yields a meromorphic extension of the incoming resolvent
$ (G -  \lambda)^{-1} \begin{pmatrix} f \\ g \end{pmatrix}.$ We call {\it incoming resonances} the poles $z,\: \Re z \geq 0$ of this incoming resolvent  and denote them by ${\rm Res}\:(G).$ Similarly,  we can define the {\it outgoing resonances} as the poles  $w, \: \Re w \leq 0,$ of the meromorphic continuation of the outgoing resolvent $(G - \lambda)^{-1}: {\mathcal H}_{comp} \rightarrow {\mathcal D}_{loc}$ which is analytic for $\Re w > 0$ and meromorphic for $\Re w \leq 0.$  

 Now assume that $a > 0$ is such that $R_D(\lambda): L^2_{comp}(\Omega) \rightarrow {\bf D}_{loc}$ and $N(\lambda)^{-1}$ are analytic for $\re \lambda < a.$  This implies that $K(\lambda)$ is also analytic for $\re \lambda < a.$ Let 
$\zeta \subset \{z\in \C: \: \re z < a\}$ be closed positively oriented curve without self intersections such that $\cc(\lambda)^{-1}$ has no poles on $\zeta.$  Let $(f, g) \in {\mathcal H}_{comp}$ and let $\varphi(x) \in C_0^{\infty}(\R^d)$ is chosen so that $\varphi \equiv 1$ on $\supp f \cup \supp g.$ From (\ref{eq:2.6}) we give
 $$\varphi u = -  \varphi R_D(\lambda)   (g + \lambda f) + \varphi K(\lambda) \cc(\lambda)^{-1} \Bigl[ \pa_{\nu}\Bigl (R_D(\lambda) ( g + \lambda f)\Bigr)\Big\vert_{\Gamma}  + \gamma f\vert_{\Gamma} \Bigr],$$   
 $$\varphi w = \lambda \varphi u + f.$$
 Since $\varphi R_D(\lambda)(g + \lambda f) $ is analytic for $\Re \lambda < a$,  the integral over $\zeta$ of this term vanishes and 
 $$\begin{pmatrix}  f \\ g \end{pmatrix}  \rightarrow  \int_{\zeta} \varphi(G -\lambda)^{-1} \begin{pmatrix}  f \\ g \end{pmatrix} d\lambda= \int_{\zeta} \begin{pmatrix} \varphi u \\ \varphi w \end{pmatrix} d\lambda $$
 $$= \int_{\zeta} \varphi\begin{pmatrix}A_{11} & A_{12}\\A_{21} & A_{22} \end{pmatrix} \begin{pmatrix}  f \\ g \end{pmatrix} d\lambda.$$
 By the cyclicity  of the trace the factor $\varphi$ can be transferred to the right and the trace of above operator  becomes
 $${\rm tr} \int_{\zeta} (A_{11} + A_{22}) d\lambda = {\rm tr}  \int_{\zeta}  K(\lambda) \cc(\lambda)^{-1} \Bigl[ \pa_{\nu}R_D(\lambda)2 \lambda + \gamma\Bigr]\Big\vert_{\Gamma}   d\lambda $$
 $$= {\rm tr}  \int_{\zeta}  \cc(\lambda)^{-1} \Bigl( \pa_{\nu}\Bigl(R_D(\lambda) 2 \lambda K(\lambda)\Bigr)\big\vert_{\Gamma} + \gamma\Bigr) d\lambda. $$ 
 In the last equality we used the fact that $K(\lambda)$ is analytic and following Lemma 2.2 in \cite{SjV} we can transfer $K(\lambda)$ to the right.  Since 
 $$(-\Delta + \lambda^2) \frac{\pa K(\lambda)}{\pa \lambda} = -2 \lambda K(\lambda),\: \frac{\pa K(\lambda)}{\pa \lambda}\Big\vert_{\Gamma} = 0,$$
  one deduces
 $$\pa_{\nu}\Bigl(R_D(\lambda) 2 \lambda K(\lambda)\Bigr)\big\vert_{\Gamma}  = -\frac{\pa N(\lambda)}{\pa \lambda}.$$
 Thus  we obtain
 \begin{prop}Let $ a > 0$ and let $\zeta \subset \{z\in \C: \: \re z < a\}$ be closed positively oriented curve without self intersections such that the  incoming resolvent $R_D(\lambda): \: L^2_{comp}(\Omega) \rightarrow {\bf D}_{loc}$  and $N(\lambda)^{-1}$ are analytic for $\re \lambda < a$. Assume that there are no poles of $\cc(\lambda)^{-1} $ on $\zeta.$  Then  the  extended incoming resolvent $(\lambda - G)^{-1}:\: {\mathcal H}_{comp} \rightarrow {\mathcal D}_{loc}$ satisfies   
 
  \begin{equation} \label{eq:2.7}
 {\rm tr} \frac{1}{2 \pi \ii} \int_{\zeta}  (\lambda - G)^{-1}  d\lambda = {\rm tr}_{H^{1/2}(\Gamma)}\frac{1}{2 \pi \ii} \int_{\zeta} \cc(\lambda)^{-1} \frac{d \cc(\lambda)}{d\lambda} d \lambda .
\end{equation}
\end{prop}

It is clear that if the resolvent $(\lambda - G)^{-1}: \: {\mathcal H}_{comp} \rightarrow {\mathcal D}_{loc}$ has a pole at $\lambda_0$, then $\cc(\lambda)^{-1}$ must have a pole at $\lambda_0.$ From (\ref{eq:2.3}) this means that $(Id - N(\lambda)^{-1} \lambda \gamma)^{-1}$ has a pole at $\lambda_0$. The fact that $N(\lambda)^{-1}$ is compact implies that for some $f \in H^{3/2}(\Gamma)$ we have 
$$(Id - N(\lambda_0)^{-1} \lambda_0 \gamma) f = 0,$$
hence
$\cc(\lambda_0) f = 0$. 
 Choosing an incoming solution of (\ref{eq:2.1}) with Dirichlet data $f$, from the last condition we obtain the existence of a solution $v \in {\mathcal D}_{loc}$ of the problem
\begin{equation} \label{eq:2.8}
\begin{cases} (-\Delta +\lambda_0^2) v = 0 \: {\rm in}\: \Omega,\\
\pa_{\nu} v - \lambda_0 \gamma(x) v= 0 \:{\rm on}\: \Gamma,\\
v : \lambda_0-{\rm incoming}.
 \end{cases}
\end{equation}   
If $ v \in \mathcal D,\: \lambda_0$ is an eigenvalues of $G$. If $v \notin {\mathcal D},$ we obtain an incoming  resonance $\lambda_0$ and a resonance state $v \in \dc_{loc}$ for which $(G - \lambda_0) v = 0$ in the sense of distributions.
In conclusion, the left hand side of (\ref{eq:2.7}) is equal to the sum of the multiplicities of the eigenvalues and the incoming resonances of $G$ included in the domain $\mathcal U$ bounded by $\zeta.$

It is easy to examine the case $\lambda \in \ii \R.$ First, let  $\lambda_0 \in \ii \R, \lambda_0 \neq 0$ and let $v$ be an incoming solution of (\ref{eq:2.8}).
 Write $v = \psi v + (1- \psi) v$ with a function $\psi \in C_0^{\infty}(\R^d)$, equal to 1 in a neighborhood of $\bar{K}$. Then $(- \Delta + \lambda_0^2) ((1- \psi)v) = F \in L^2(\R^d)$ with $F =  0$ for $|x| > a_0$  and some $a_0 > 0.$ Since $(1 - \psi)v$ is incoming, by Corollary 4.3 in \cite{LP} (see also Theorem 4.17 in \cite{DZ}, where outgoing solutions are incoming ones in our sense), we deduce
$(1- \psi) v = 0$ for $|x| > a_0$ and this yields $v \in {\mathcal D}.$ The operator $G$ has no eigenvalues on $\ii \R$, hence  $v = 0.$ We will use this result in the proof of Proposition A.4 in Appendix. Second, let $\lambda_0 = 0$. Then $\Delta v = 0$ and $\pa_{\nu} v \vert_{\Gamma} = 0$. Since $v$ is incoming, we have $v = \oc(r^{2- d}), \: \pa_r v = \oc(r^{1-d})$ as $|x| = r \to \infty.$ Applying the proof of Theorem 4.19 in \cite{DZ}, we conclude that $v = 0.$

  For the analysis in Section 5 it is convenient to replace $\cc(\lambda)$ by the operator $- \cc(\lambda) = -N(\lambda) + \lambda \gamma$ and write (\ref{eq:2.7}) with $-\cc(\lambda).$ Next setting $\tilde{\cc}(\lambda) = - \frac{N(\lambda)}{\lambda} + \gamma$, for contour $\zeta$ and domain ${\mathcal U}$ not including 0, we obtain
 \begin{equation*} 
 {\rm tr}\frac{1}{2 \pi \ii} \int_{\zeta} \cc(\lambda)^{-1} \frac{d \cc(\lambda)}{d\lambda} d \lambda =  {\rm tr}\frac{1}{2 \pi \ii} \int_{\zeta}\tilde{\cc}(\lambda)^{-1} \frac{d \tilde{\cc}(\lambda)}{d\lambda} d \lambda .
 \end{equation*}

Now we pass to semi-classical parametrisation $\th = \frac{\ii}{\lambda}, \:0 <  \Re \th \ll 1 $ and  introduce the operator $C(\th) : = \ii\th \nc(\ii\th^{-1} ) + \gamma= - \tilde{\nc}(\th) + \gamma$. We have $\Im \lambda > 0 $ and  consider a contour $\zeta \subset \{z\in \C:\: \Im z > 0\}.$ Applying the above equality, the trace formula (\ref{eq:2.7}) becomes
\begin{equation} \label{eq:2.9}
 {\rm tr}\: \frac{1}{2 \pi \ii} \int_{\zeta} (\lambda - G)^{-1} d\lambda = {\rm tr} \frac{1}{2 \pi \ii} \int_{\tilde{\zeta}} C(\th)^{-1}  \dot{C}(\th) d \th,
\end{equation} 
where $\dot{C}$ denote the derivative with respect to $\th$ and $\tilde{\zeta}$ is the curve $\tilde{\zeta} = \{z \in \C:  z = \frac{\ii}{w}, \: w  \in \zeta\}.$ The eigenvalues and incoming resonances are symmetric with respect to real axis, since if $v$ is a solution of (\ref{eq:2.8}) with $\Im \lambda_0 > 0$, then $\bar{v}_0$ is a solution with $\bar{\lambda}_0.$ On the other hand, according to Theorem 1.1 in \cite{P1},in the case $0 < \gamma(x) < 1$ for fixed $d_0 > 0$ there are only finite number eigenvalues $\lambda$ with $|\Im \lambda| \leq d_0.$ Hence it is sufficient to study the eigenvalues in the half plane $\{ \Im \lambda > 0\}.$  
Clearly, $\Re \th = \frac{\im \lambda}{|\lambda|^2},\: \im \th = \frac{ \re \lambda}{|\lambda|^2}.$ For $\lambda \in \Lambda$ we obtain $|\im \th| \leq C_1 |\th|^4$ and we will work with $\th \in L$, where
\begin{equation} \label{eq:2.10}
 L := \{\th \in \C: \: |\im \th | \leq C_0 |\th|^2, \: 0 < |\th | \leq h_0\}.
\end{equation}
The operator $\tilde{\nc}(\th)$ is related to the problem
\begin{equation} \label{eq:2.11}
\begin{cases} (\th^2\Delta +1) u = 0 \:{\rm in}\: \Omega,\\
u = f \: {\rm on}\: \Gamma.\\
u -{\rm incoming}. \end{cases} 
\end{equation}

For the analysis of the location of eigenvalues and incoming resonances in Section 4 it is more convenient to work with another parametrisation $\lambda = \frac{\ii \sqrt{z}}{h}$ with
$z = 1 + \ii \Im z,\: |\im z| \leq 1,\:0 < h \leq h_0$ and $\re \sqrt{z} > 0.$ Thus we have again $\im \lambda > 0.$ The problem (\ref{eq:2.1}) becomes
\begin{equation} \label{eq:2.12}
\begin{cases} (-h^2\Delta - z) u = 0 \:{\rm in}\: \Omega,\\
u = f \: {\rm on}\: \Gamma,\\
u -{\rm incoming}. \end{cases} 
\end{equation}
 and we introduce the operator $\nc(z, h)f = -\ii h \pa_{\nu} u\vert_{\Gamma}.$ We have 
 $$\cc(f) = \frac{\ii}{h} \Bigl(\nc(z, h) f - \sqrt{z} f\Bigr)$$
and the equation $\cc(f) = 0$ yields
\begin{equation}\label{eq:2.13}
(\nc(z, h) - \sqrt{z}) f = 0.
\end{equation}

 Now we recall some definition concerning semi-classical wave fronts sets.  Given a manifold $X$ with dimension $d-1,$ consider $\widetilde{T^*(X)} = T^*(X) \cup S^*(X),$ where $S^*(X) \simeq \{(x, \infty \xi): (x, \xi) \in S^*(X)\} .$ The point in $T^*(X)$ will be called finite and the points in $S^*(X)$ will be called infinite (see \cite{A}, \cite{G}, \cite{SjV}). We are interested of semi-classical distributions $u(x, h) \in {\mathcal D}'(X),\: 0 < h \leq h_0$ for which
    $$\forall \chi \in C_0^{\infty}(X), \: |\widehat{(\chi u)} (\xi)| \leq C_N h^{-N} (1 +  |\xi |)^{N}\: {\rm for\: some}\: N,$$  
  $\hat{u}$ being the semi-classical Fourier transform
  $$\hat{u}(\xi, h) = \int e^{-\frac{\ii \la x, \xi\ra}{h} }u(x, h) dx.$$
   Let $\rho = (x_0, \xi_0) \in T^* (X)$. Then $ \rho \notin \widetilde{WF}(u)$ if there exists $\psi \in C_0^{\infty}(X)$ and $\zeta(\xi) \in C_0^{\infty}(\R^{d-1})$  with $\psi(x_0) = 1,\: \zeta(\xi_0) = 1 $  such that $|\zeta(\xi)\widehat{\psi u}(\xi)| \leq C_N h^N, \: \forall N.$ Similarly, an infinite point $\rho = (x_0, \infty\xi_0) \in S^*(X)$ is not in $\widetilde{WF}(u)$ if there exists $\psi(x) \in C_0^{\infty}(X)$ with $\psi(x_0) = 1$ such that 
  $$|\widehat{\psi u}(\xi)| \leq C_N h^N( 1+ |\xi|)^{-N},\: \forall N$$
  for all $\xi, |\xi | > C$  in some conic neighborhood of $\xi_0.$ Next as in \cite{SjV}, we introduce the the space of symbols $a(x, \xi; h) \in S^{m, k}(\Gamma \times \R^{d-1}\times (0, h_0])$ such that
  $$|\pa_x^{\alpha} \pa_{\xi}^{\beta} a(x, \xi; h) | \leq C_{\alpha, \beta}h^{-k}\la \xi\ra^{m- |\beta|},\: \forall (x,\xi) \in \Gamma \times \R^{d-1}, \: \forall \alpha, \forall \beta. $$  
  Let $S^m_{\rho, \delta}(\Gamma \times \R^{d-1})$ be the class of symbols $a(x, \xi) \in C^{\infty}(\Gamma \times \R^{d-1})$ such that 
  $$|\pa_x^{\alpha} \pa_{\xi}^{\beta} a(x, \xi) | \leq C_{\alpha, \beta}\la \xi \ra^{m-\rho |\beta| + \delta |\alpha|},\:\forall (x,\xi) \in \Gamma \times \R^{d-1},\: \forall \alpha, \forall\beta $$    
  and let  $S^{m, k}_{cl} \subset   S^{m, k}$ be the class of symbols $a(x, \xi; h)$ having asymptotic expansion
  $$a(x, \xi; h) \sim \sum_{j = 0}^{\infty} h^{j-k} a_j(x, \xi)$$
  with $a_j \in S_{1, 0}^{m- j}(\Gamma \times \R^{d-1})$. Denote by $\Op_h(a)$ the $h$-pseudo-differential operator
  $$(\Op_h(a) f)(x) =
 (2 \pi h)^{-d + 1}\int_{T^*X} e^{\ii \langle x - y, \xi \rangle /h} a(x, \xi; h) f(y) dy d \xi.$$ 
 Set $S^{-\infty, \-\infty} = \bigcap _{m, k} S^{m, k}$ and introduce the spaces of  $h$-pseudo-differential operators $L^{m, k}, \: L^{m, k}_{cl}$ with symbols in $S^{m, k},\: S^{m, k}_{cl},$ respectively. 
 
 Passing to semi-classical wave fronts of operators $A \in L^{m, k}$, consider the compactification $\widetilde{T^*(\Gamma)} = T^*(\Gamma) \cup S^*(\Gamma).$ We say that $\rho = (x_0, \xi_0)\notin \widetilde{WF}(A)$ if the symbol of $A$ in a neighborhood of $\rho$  is in the class $S^{-\infty, -\infty}.$  For $\rho = (x_0, \infty \xi_0) \in \widetilde{T^*(\Gamma)} \setminus T^*(\Gamma),$ we have $\rho  \notin \widetilde{WF}(A)$ if the symbol of $A$ is in the class $S^{-\infty, -\infty}$ in the set $\{(x, \xi):\: x \in U_0,\: \frac{\xi}{|\xi|} \in V_0, \: |\xi | \geq C\}$ with $U_0, V_0$ being neighborhoods of $x_0, \xi_0$, respectively. 
 
\section{Parametrix in the hyperbolic region} 
\subsection{Parametrix for $\nc(z, h)$}

\renewcommand{\Re}{\mathop{\rm Re}\nolimits}
\renewcommand{\Im}{\mathop{\rm Im}\nolimits}
\numberwithin{equation}{section}

\def\bbbone{{\mathchoice {1\mskip-4mu \text{l}} {1\mskip-4mu \text{l}}
{ 1\mskip-4.5mu \text{l}} { 1\mskip-5mu \text{l}}}}
\def\phi{\varphi}
\def\epsilon{\varepsilon}
\def\kappa{\varkappa}
\def\eT{e^{-\lambda T}}
\def\ii{{\bf i}}
\def\hh{\hat{x}}
\def\hhx{\hat{\xi}}
\def\12{\frac{1}{2}}
\def\Rc{{\mathcal R}}
\def\tE{\tilde{E}}
\def\ep{\epsilon}
\def\la{\langle}
\def\ra{\rangle}
\def\co{{\mathcal O}}
\def\pa{\partial}
\def\h{^}
\def\pa{\partial}
\def\qp{Q^{+}_{\delta}}
\def\qm{Q^{-}_{\delta}}
\def\qn{Q^{0}_{\delta}}
\def\th{\tilde{h}}

Throughout this and following sections we assume that the obstacle $K$  is strictly convex. In this subsection we use the parametrisation  $\lambda  = \frac{ \ii \sqrt{z}}{h},\: 0 < h \ll 1,\:  z = 1 + \ii  \theta$, while in the subsection 3.2 we will work with $\lambda = \frac{\ii}{\th},\: \th \in L.$  Moreover, we assume that $\theta$ satisfies the inequalities 
$$ -c h |\log h| \leq \theta \leq h^{\ep},\: c > 0,\:0 < \ep \ll 1.$$
 Introduce local geodesic coordinates $(x_1, x'),\: x' \in {\mathcal V} \subset \R^{d-1}$, where  the boundary $\Gamma$ locally is given by $x_1 = 0$ and $x_1 > 0$  in the domain $\Omega.$ In these coordinates one has
 $$- h^2\Delta = h^2D_{x_1}^2 + Q(x, hD_{x'}) + h a_0(x) D_{x_1} ,\: D_{x_j} = - \ii \pa_{x_j}, \: j = 1,...,d$$
 with second order operator $Q$ with symbol $Q = r_0(x', \xi') - x_1q(x, \xi').$ Here
 $$r_0(x', \xi') = \langle B(x') \xi', \xi' \rangle \geq \alpha_0 |\xi'|^2, \: \alpha_0 > 0,\: \forall \xi  \in \R^{d-1}$$
 and $q(x', \xi') \geq q_0 |\xi'|^2, \: q_0 > 0.$ Define the hyperbolic ${\mathscr H}$, glancing ${\mathscr G}$ and elliptic regions ${\mathscr E}$ by
 $${\mathscr H} = \{(x', \xi') \in T^*(\Gamma):\:1 - r_0(x', \xi')  > 0\},$$
 $${\mathscr G} = \{(x', \xi') \in T^*(\Gamma):\:1 - r_0(x', \xi')  =  0\},$$
 $${\mathscr E} = \{(x', \xi') \in T^*(\Gamma):\:1 - r_0(x', \xi') < 0\}.$$

 For $0 < \delta \ll 1$ consider the function
 $\chi^{-}_{\delta}(x', \xi') = \psi^{-}\Bigl( \frac{r_0(x', \xi') - 1}{\delta}\Bigr)$, where 
 $$\psi^{-} \in C^{\infty} (\R; [0,1]),\: \supp \psi^{-} \subset (-1, -\infty)$$ and $\psi^{-}(t) \equiv 1$ for $t \leq - 3/2$. Let $\zeta_0 =(x'_0, \xi'_0) \in \supp \chi^{-}_{\delta}$ and let ${\mathcal U}$ be a small neighborhood of $\zeta_0$ contained in the set $\{(x', \xi') \in T^*(\Gamma): \: 1 - r_0(x', \xi') \geq \delta/2\} \subset {\mathscr H}.$ Choose $\psi(x', \xi') \in C_0^{\infty}({\mathcal U})$ such that $\psi = 1$ in a neighborhood of $\zeta_0$ and introduce the symbol $\rho(x', \xi') := \sqrt{z - r_0(x', \xi')}.$
 
 Our purpose is to construct a  parametrix
for the problem
\begin{equation} \label{eq:3.1}
\begin{cases} (- h^2\Delta - z) u= 0, \: x \in \Omega,\\
u = \Op_h(\chi^{-}_{\delta}) f, \: x \in \Gamma,\\
u-\frac{\ii \sqrt{z}}{h}-\rm{incoming},\\
\end{cases}
\end{equation}
with $f \in L^2(\Gamma)$ and to obtain an approximation for the operator 
$$\nc(z, h)\Op_h(\chi^{-}_{\delta})f = -\ii h \pa_{\nu} u\vert_{\Gamma}.$$
 The local parametrix of (\ref{eq:3.1}) is $u^{-}_{\psi} = \Phi(x_1) {\mathcal K}^{-}_{\psi} f$ with
\begin{equation} \label{eq:3.2}
({\mathcal K}^{-}_{\psi}  f) (x) =   (2 \pi h)^{- d + 1} \iint e^{\frac{\ii}{h}( \la y', \xi'\ra + \varphi(x, \xi', \theta))} a(x, \xi', \theta, h) f(y') dy'd\xi',
\end{equation} 
where $\Phi(x_1) = \chi_0(\frac{x_1}{\delta_0}),\: \chi_0 \in C_0^{\infty}(\R)$ is such that $\chi_0(t) = 1$ for $|t| \leq \delta_0, \: \chi_0(t) = 2$ for $ |t| \geq 2,\: 0 < \delta_0 \ll 1.$ 
This incoming condition will be arranged later by applying incoming resolvents to local parametrix.

We follow the construction in Section 4,  \cite{V3} and for convenience of the reader we present some details.  The main difference with \cite{V3} is that we treat the case $-c h|\log h| \leq \theta \leq h^{\ep}$ for strictly convex obstacles, while in \cite{V3} the analysis was given for $h^{1- \ep} \leq \theta \leq h^{\ep}$ for obstacles with arbitrary geometry. The symbol $a(x, \xi', \theta, h)$ will have support for $(x', \xi') \in {\mathcal U}$, so in the construction below the condition $1 - r_0(x', \xi') \geq \delta/2$ holds. The phase $\varphi$ satisfies 
\begin{equation} \label{eq:3.3}
\varphi_{x_1}^2 + \la B(x) \varphi_{x'}, \varphi_{x'}\ra = 1 + \ii \theta + \theta^M {\mathcal R}_M, \: \varphi\vert_{x_1 = 0} = - \la x', \xi' \rangle
\end{equation}
 and has the form
$$\varphi = \sum_{j = 0}^{M- 1} (\ii \theta)^j\varphi_j(x, \xi')$$
with real valued phase functions $\varphi_j$. The function $\varphi_0$ is  a local solution of the problem
\begin{equation} \label{eq:3.4}
\begin{cases}
$$\pa_{x_1} \varphi_0 =  \sqrt{1 -  \la B(x) \nabla_{x'} \varphi_0, \nabla_{x'} \varphi_0\ra},\\
\varphi_0\vert_{x_1 = 0}= - \la x', \xi'\ra
\end{cases}
 \end{equation}
existing for small $x_1$ and 
$$\pa_{x_1} \varphi_0\vert_{x_1 = 0} = \sqrt{1 - r_0(x', \xi')},$$
where $ 1 - r_0 \geq \delta/2.$ The sign of $\sqrt{1 - r_0(x', \xi')}$ determines the propagation of singularities in the interior of $\Omega$.
The functions $\varphi_j,\: 1 \leq j \leq M-1,$ satisfy the equations
$$\sum_{j= 0}^k \pa_{x_1} \varphi_j \pa_{x_1} \varphi_{k-j} + \sum_{j= 0}^k \la B(x) \nabla_{x'} \varphi_j \nabla_{x'} \varphi_{k-j} \ra = \ep_k,\: \varphi_k\vert_{x_1= 0} = 0 ,\:1 \leq k \leq M-1$$
with $\ep_1 =1, \: \ep_k = 0, \: k \geq 2$ and the remainder ${\mathcal R}_M$ is bounded uniformly with respect to $\theta.$ From the above equations with $k = 1$ one obtains
\begin{eqnarray} 
\theta \pa_{x_1} \varphi_1\vert_{x_1 = 0} =  \frac{\theta}{2 \sqrt{1- r_0}} \geq  \frac{\theta}{2},\: \rm{for}\: \theta \geq 0,\\
\theta \pa_{x_1} \varphi_1\vert_{x_1 = 0} \geq  - \frac{c h |\log h| }{\sqrt{2\delta}},\: \rm{for} \: \theta < 0.
\end{eqnarray} 
 This implies
 $$ \im \pa_{x_1} \varphi \vert_{x_1 = 0} = \theta \pa_{x_1}\varphi_1\vert_{x_1 = 0} + \co (\theta^2) \geq \frac{\theta}{3},\: \theta \geq 0 $$
and  for small $\theta$ and $x_1$ we have
 $$ \im \varphi= \theta \varphi_1 + \co(x_1 \theta^2) \geq \frac{x_1 \theta}{4}, \: \theta \geq 0,$$
  $$ \im \varphi  \geq -\frac{c x_1 h|\log h|}{2\sqrt{2\delta}},\: \theta < 0.$$ 
 The eikonal equation (\ref{eq:3.3}) yields
 $$(\pa_{x_1} \varphi\vert_{x_1 = 0})^2 = \rho^2 (1 + \co(\theta^M)),$$
 hence for $x_1$ small enough
 \begin{equation} \label{eq:3.7}
 \pa_{x_1} \varphi\vert_{x_1 = 0} = \rho + \co(\theta^{M/2}).
 \end{equation}
 
 The amplitude has the form
 $$a = \sum_{k = 0}^{m} h^k a_k(x, \xi', \theta),$$
 where the functions $a_k$ satisfy the transport equations
 $$2 \ii \pa_{x_1} \varphi \pa_{x_1} a_k + 2 \ii \la B(x) \nabla_{x'} \varphi, \nabla_{x'} a_k\ra + \Delta a_{k-1} = \theta^MQ_{M}^{(k)},\: 0 \leq k \leq m,$$
 $$a_0\vert_{x_1 = 0} = \psi,\: a_k\vert_{x_1 = 0} = 0, \: k \geq 1, \: a_{-1} = 0.$$
 We search the  functions $a_k$ in the form
 $$ a_k = \sum_{j= 0}^{M -1} (\ii \theta)^k a_{k, j} (x, \xi', \theta),$$
 with $a_{0, 0}\vert_{x_1 = 0} = \psi, \: a_{k, j} \vert_{x_1 = 0} = 0$ for $k + j \geq 1.$ 
 We refer to Section 4, \cite{V3} for the determination of $a_{k, j}$.

The construction of ${\mathcal K}^{-}_{\psi} $ implies
$$(h^2 \Delta + z) u^{-}_{\psi} = {\mathcal K}_{1, \psi}^{-} f +  {\mathcal K}^{-}_{2, \psi} f.$$
Here 
$$ {\mathcal K}^{-}_{1, \psi}f = \Bigl [h^2 \Bigl( 2 \Phi'(x_1)\pa_{x_1} + \Phi''(x_1)\Bigr) + h a_0(x) \Phi'(x_1)\Bigr] {\mathcal K}^{-}_{\psi}f,$$
$$({\mathcal K}^{-}_{2, \psi} f)(x) = (2\pi h)^{-d +1} \iint e^{\frac{\ii}{h}( \la y', \xi' \ra + \varphi(x, \xi', \theta))} A^{-}_{2,\psi} (x, \xi', \theta, h) f(y') dy'd\xi'$$
with
\begin{equation} \label{eq:3.8}
A^{-}_{2, \psi} = \Phi(x_1) \Bigl(\theta^M{\mathcal R}_{M} a + \theta ^M \sum_{k = 0}^m h^{k +1} {\mathcal Q}_{M}^{(k)} + h^{m+2} \Delta a_m\Bigr).
\end{equation}
By using a partition of unity on $\supp \chi^{-}_{\delta}$ with functions $\psi_j$, we arrange $\sum_{j= 1}^J {\psi_j} = \chi^{-}_{\delta}$ and consider the parametrix
$w = \sum_{j= 1}^J u^-_{\psi_j}.$ Set 
$${\mathcal K}^{-}_{k} = \sum_{j = 1}^J {\mathcal K}^{-}_{k, \psi_j}, \: k = 1, 2,\: A^{-}_2 = \sum_{j = 1}^J A^{-}_{2, \psi_j}.$$

Let $R_0^{-}(z, h) = (-h^2\Delta_0 - z)^{-1}$ be the incoming resolvent of the free Laplacian $\Delta_0$ in $\R^d$ witch for $ \im z > 0$ is an analytic operator valued function  bounded from $L^2(\R^d)$ to $H^2(\R^d).$  Let $\Psi\in C_0^{\infty}(\R^d)$ be a cut-off function such that $\Psi = 1$ on a small neighborhood of $K$. Then  the cut-off resolvent $\Psi (-h^2\Delta_0 - z) ^{-1} \Psi$ is analytic in $\C.$
Introduce the semi-classical Sobolev spaces $H_h^s(X)$ with semi-classical norm 
$$\|f\|_{H^s_h(X)} = \|\la h D\ra^s f\|_{L^2(X)},\: \la h D \ra : = ( 1 - h^2\Delta_X))^{1/2}.$$

For  $\re \sqrt{z} \geq 1$ the sut-off resolvent satisfy the estimates (see for instance Theorem 3.1 in \cite{DZ} and recall Remark 2.1) .
\begin{equation} \label{eq:3.9} 
\|\Psi R_0^{-} (z, h)  \Psi\|_{L^2(\R^d) \to H^j_h(\R^d)} \leq C_jh^{-1} ( h+ \sqrt{|z|})^{j- 1} e^{\frac{L}{h} (\im z)_{-}},\: j = 0, 1, 2,
\end{equation} 
where $L > \sup \{|x - y|: x, y \in \supp \Psi\}$ and $x_{-} = \max \{0, -x\}.$ In particular, for $\im z > 0$ the exponential factor on the right hand side is equal to 1, while for 
$-ch |\log h| \leq \im z < 0$ this factor is bounded by $h^{-c L}$.\\

Similarly, introduce the incoming resolvent $R^{-}(z, h)= (-h^2 \Delta_D - z)^{-1}$, where $\Delta_D$ is the Laplacian with Dirichlet boundary condition on $\Gamma$ and domain
${\mathcal D} =H^2(\Omega) \cap H^1_0(\Omega).$ Then $R^{-}(z, h): L^2(\Omega) \longrightarrow {\mathcal D}$ for $\im z > 0$ and $\Psi R^{-}(z, h) \Psi$ admits a meromorphic continuation in $\C$ with poles in $\{z \in \C: \Im z < 0\}.$ Since the obstacle $K$ is non-trapping,  for 
$$\re \sqrt{z} \geq 1,\:  \im z \geq -c h |\log h|,\:0 < h \leq h_0$$
 we have the estimates (see Theorem 4.43 in \cite{DZ} and Theorem 2 in Chapter X, \cite{Va})
\begin{equation} \label{eq:3.10}
\|\Psi R^{-}(z, h) \Psi\|_{L^2(\Omega) \to H^{j}_h(\Omega)} \leq C_j h^{-1}e^{T \frac{(\Im z)_{-}}{h}}, C_j > 0, \: T > 0, \: j = 0, 1, 2.
\end{equation} 

To build a global parametrix, consider
$$\tilde{u} = w - R_0^{-}(z, h) {\mathcal K}^{-}_1f - R^{-} (z, h)  {\mathcal K}^{-}_2f.$$
The operator $K^{-}_2$ is bounded from $L^2(\Gamma)$ to $L^2_{loc}(\Omega).$ To prove this, consider $K^{-}_2$ as a $h$-Fourier integral operator with real phase function
$\la y', \xi'\ra - \re \varphi(x, \xi', \theta)$ and amplitude $b(x, \xi', \theta, h) = e^{-\frac{\im \varphi}{h}} A_2^{-}(x, \xi', \theta, h)$
depending on the parameter $x_1 \in [0, 2\delta_0].$ Write 
 $$b = \exp \Bigl(-\frac{x_1 \theta}{2h} \Bigl( \frac{1}{\sqrt{1 - r_0}} + \co(x_1) + \co(\theta)\Bigr)\Bigr) A_2^{-} = \exp \Bigl(-\frac{x_1 \theta}{2h} g\Bigr) A_2^{-}.$$
Therefore
$$\pa_{x'}^{\alpha} b = \sum_{|\beta_1| + |\beta_2| = |\alpha| } C_{\beta_1, \beta_2}  \Bigl(\frac{x_1 \theta}{2h}\Bigr)^{|\beta_1|} \exp \Bigl(-\frac{x_1 \theta}{2h} g\Bigr) (\pa_{x'}^{\beta_1} g)(\pa_{x'}^{\beta_2} b_{\beta_2})$$
with some symbol $b_{\beta_2}.$
For small $0 \leq x_1\leq 2 \delta_0$ and small $h$ the product $$\Bigl(\frac{x_1 \theta}{2h}\Bigr)^{|\beta_1|} \exp \Bigl(-\frac{x_1 \theta}{2h} g\Bigr), \: \theta \geq 0$$
 is bounded since $g \geq 1/2.$  For $-c h |\log h| \leq \theta < 0$ the product 
 $$\Bigl(\frac{c|\log h|x_1 \theta}{\sqrt{2\delta} }\Bigr)^{|\beta_1|} \exp \Bigl(-\frac{c |\log h| x_1 \theta}{ \sqrt{2\delta}} g\Bigr), \: \theta \geq 0$$ 
 is bounded by $(C_1 |\log h|)^{|\beta_1|} h^{- C_2}$ with $C_ 1 > 0, C_2 > 0$ independent of $h$. Similarly, we estimates the derivatives $\pa_{\xi'}^{\beta} b.$ Consequently,
 taking into account (\ref{eq:3.8}), we have
 $$|\pa_{x'}^{\alpha} \pa_{\xi'}^{\beta} b| \leq C_{\alpha, \beta} |\log h|^{|\alpha| + |\beta|} h^{-C_2}\Bigl(\co(h^{\ep M}) + \co(h^{m+ 2})\Bigr), \: C_2 > 0$$
  with large $M$ and $m$. For $\theta \geq  0$ the factor  $|\log h|^{|\alpha| + |\beta|} h^{-C_2}$ must be replaced by 1. On the other hand,  for small $x_1$ the phase $\la y', \xi'\ra - \re \varphi(x_1, \xi', \theta)$ is non-degenerate and
 $$\Bigl | \det \Bigl(\frac{\pa^2 \re \varphi}{\pa x'\pa \xi'}\Bigr) \Bigr | \geq D > 0$$
 because $\varphi\vert_{x_1 = 0} = -\la x', \xi'\ra.$ By a standard argument we consider the $h$-pseudo-differential operator $(K_2^{-})^* K_2^{-}$ with parameter $x_1$ and  deduce the estimate 
 \begin{equation} \label{eq:3.11}
  \|K_2^{-} f\|_{L^2(\Omega)} \leq A_N h^N \|f\|_{L^2(\Gamma)}
  \end{equation} 
  with big $N$ choosing $M$ and $m$ large and depending on $N$. Applying the estimates (\ref{eq:3.10}), we conclude that $\|R(z, h) K_2^{-} f\|_{H^{3/2}_h(\Gamma)} = \co(h^{N-3/2})\|f \|_{L^2}$  
exploiting the operator of restriction 
$$\gamma_0 = \co(h^{-1/2}):  H^2_h(\Omega) \rightarrow H^{3/2}_h(\Gamma) .$$

    To deal with the term $R_0^{-}(z, h) {\mathcal K}^{-}_1f$,  we will apply the argument in  Appendix A.II.1, \cite{G} for  ${\mathcal K}^{-}_1 f = F(h) $. For this purpose repeating the proof of Corollary A.II.4, \cite{G}, one proves that the semi-classical wave front set $\widetilde{WF}(F(h))$  is included in the intersection of the set of outgoing rays issued from $\{(y', \eta') \in T^*(\Gamma):\: (y', \eta') \in \supp \chi^{-}_{\delta}\}$ with a bounded set $\mathcal W$ such that ${\rm dist}\: ( {\mathcal W}, K  ) \geq \ep_0 > 0.$ The set $\mathcal W$  is determined by the support of the derivatives  $\Phi^j(x_1),\: j = 1,2$, while the wavefront $\widetilde{WF}(K^{-}_{\psi})$  is determined by the set $\{(x, \varphi_x):\: (x', \xi') \in \supp \chi^{-}_{\delta} \}$ and the fact that the phase $\varphi$ is chosen so that $\varphi_{x_1} > 0$ (see \cite{G}, \cite{A}). Finally, $\widetilde{WF}(R_0^{-}(z, h) F(h))$ is given by the outgoing rays issued from $\widetilde{WF}(F(h))$ (By Remark 2.1,  the resolvent $R_0^{-} (z, h)$ is outgoing in the sense of  \cite{G}). Since $K$ is strictly convex, these rays don't meet the boundary $\Gamma$ and 
    $$\|R_0^{-}(z, h) F(h)\|_{H^m_h({\mathcal O})} = \co_m(h^{\infty})\|f\|_{L^2},\: \forall m \in N $$  
    in a small neighborhood $\co$ of $K$. We conclude that $u - \tilde{u}$ is a solution of the problem
   \begin{equation} \label{eq:3.12}
\begin{cases} (- h^2\Delta - z)(u- \tilde{u})= 0, \: x \in \Omega,\\
u - \tilde{u} = -(R_0(z, h) F(h))\vert_{\Gamma}, \: x \in \Gamma,\\
(u- \tilde{u})-\rm{incoming}.\\
\end{cases}
\end{equation} 
Therefore $\|\nc(z, h)(\Op_h(\chi^{-}_{\delta})f - \tilde{u}\vert_{\Gamma})\|_{H^1_h} \leq C_N h^N \|f\|_{L^2},\: \forall N$ and 
$$\nc(z, h) \tilde{u}\vert_{\Gamma} = \nc(z, h) w\vert_{\Gamma} + \co(h^N).$$
By our construction we obtain $\nc(z, h)w\vert_{\Gamma} = T_N(z, h) \Op_h(\chi^{-}_{\delta})f $ with a $h$-pseudo-differential operator $T_N(z, h)$ having principal symbol
$\sqrt{z - r_0} $ and
\begin{equation} \label{eq:3.13}
\|(\nc(z, h) - T_N(z, h)) \Op_h(\chi^{-}_{\delta}) \|_{L^2 \to H_h^{1}} \leq C'_N h^{N}.
\end{equation}      
\subsection{Parametrix for $\tilde{\nc}(\th)$}

 In the trace formula (\ref{eq:2.9}) we have integration with respect to $\th = \frac{\ii}{\lambda}$ and it is convenient to have a parametrix holomorphic with respect to $\th.$ Comparing with the parametrisation $\lambda = \frac{\ii \sqrt{z}}{h},$ used in the first part, we give
$\re \th = \frac{\re \sqrt{z}}{|z|}h$ and $\re \th = h$ if and only if $\re\sqrt{z}= |z|.$ This leads to difficulties if we wish to extend the approximation $T_N(z, h)$ as a holomorphic function of $\th$ modulo some remainder. To overcome this problem, we will construct another parametrix in the hyperbolic region for 
$\th \in L,$ where $L$ is defined by (\ref{eq:2.10}).
Setting $\re \th = h$, for small $h_0$ we have $|\im \th | \leq C_1 h^2.$

      Next we repeat without changes the construction in Appendix A2, \cite{StV} and search a parametrix for the problem
   \begin{equation} \label{eq:3.14}
\begin{cases} (- \th^2 \Delta - 1) u= 0, \: x \in \Omega,\\
u = \Op_{\th}(\chi^{-}_{\delta}) f, \: x \in \Gamma,\\
u- \rm{incoming}.\\
\end{cases}  
\end{equation} 
Here $\Op_{\th}(\chi^{-}_{\delta})$ is a pseudo-differential operator with large parameter $\frac{1}{\th}$ having the form
$$(\Op_{\th} (\chi^{-}_{\delta}) f)(x') = (2 \pi \th)^{-d+ 1}\iint e^{\frac{\ii }{\th} \la y' - x', \xi'\ra} \chi^{-}_{\delta} (x', \xi') f(y') dy' d\xi'.$$
The symbol $\chi^{-}_{\delta}$ has compact support with respect to $\xi'$ and the above operator is well defined for $\th \in L$ (see Appendix A1, \cite{StV}) since
\begin{equation} \label{eq:3.15}
\Bigl|\frac{1}{\th} - \frac{1}{h}\Bigr|= \Bigl|\frac{ \im \th}{h \th}\Bigr | \leq C_1. 
\end{equation}
Notice that for $\lambda \in  \Lambda$ we study the problem (\ref{eq:2.11}).

 The local parametrix in local geodesic coordinates $(x_1, x')$ introduced above has the form $\Phi(x_1) H_N(\th),$ where
$$(H_N(\th) f)(x) =  (2 \pi \th)^{-d + 1} \iint e^{\frac {\ii}{\th} (\la y', \xi'\ra + \psi(x, \xi'))} a(x, \xi', \th) f(y') dy' d\xi'.$$  
The phase function $\psi$ is real valued and satisfies the equation
\begin{equation} \begin{cases} |\nabla_x \psi|^2 = 1,\\
\psi\vert_{x_1 = 0} = - \la x', \xi'\ra. \end{cases} \label{eq:3.16}
\end{equation}
In local coordinates used in the previous construction we have
$$(\pa_{x_1} \psi)^2 + \la B(x) \nabla_{x'}\psi, \nabla_{x'} \psi \ra = 1.$$
The phase $\psi$ is determined as  $\varphi_0$ above and we obtain $\pa_{x_1} \psi\vert_{x_1 = 0} = \sqrt{1 - r_0(x', \xi')} \geq \sqrt{\delta}$  for $(x', \xi') \in \supp\: \chi^{-}_{\delta}.$ The amplitude has the form
$$a = \sum_{j = 0}^{N- 1} a_j(x,\xi') \th^{j}$$
and  $a_j(x, \xi')$ are determined as solutions of the transport equations
$$ 2 \ii \la \nabla \psi, \nabla a_j \ra + \ii (\Delta \psi) a_j  = - \Delta a_{j-1},\: j = 0,...,N-1,$$
with conditions $a_0\vert_{x_1 = 0} = \chi^{-}_{\delta}(x', \xi'), \: a_j\vert_{x_1 = 0} = 0,\: j \geq 1$
and $a_{-1}= 0.$ By using (\ref{eq:3.15}), we may write $H_N(\th)$ as a Fourier integral operator with real valued phase $\frac{\la y', \xi'\ra + \psi(x, \xi')}{h}$ making the factor  
$$\exp\Bigl(\ii (\frac{1}{\th} - \frac{1}{h}) (\la y', \xi'\ra + \psi(x, \xi'))\Bigr)$$
in the amplitude. Thus as in Corollary A.II.8 in \cite{G} and (A.19) in \cite{StV}, we obtain 
$$\widetilde{WF}(H_N(\th)) \subset  \Big\{ (x, \xi, y', \eta') \in T^*(U \setminus \Gamma) \times T^*(\Gamma):\: \|\xi\|_{x}= 1, $$
$\:\:\:\:\:\:\: \:\:\:$ $(x, \xi)$  belongs to the outgoing ray issued from $(y', \eta') \in \supp \chi^{-}_{\delta}\Big\}.$\\
Here $U \subset \R^d$ is a small neighborhood of $\Gamma$ and $\|\xi\|_x$ is the norm of dual variable induced by local coordinates. Next we construct a global parametrix by using a partition of unity, the incoming free resolvent $(-\th^2 \Delta_0 - 1)^{-1} $ and the incoming resolvent $(-\th^2 \Delta_D - 1)^{-1}$ of the Dirichlet Laplacian $\Delta_D.$ The operator
$\tilde {\nc}(\th) \Op_{\th}(\chi^{-}_{\delta}) f = -\ii\th^{-1} \pa_{\nu} u\vert_{\Gamma}$ (see Section 2) has an approximation by a $\th$-pseudo-differential operator $\tilde{T}_N(\th)\Op_{\th}(\chi^{-}_{\delta})$ and $\tilde{T}(\th)$ has principal symbol $\sqrt{1 - r_0(x', \xi')}$. Similarly to (\ref{eq:3.13}) we get
\begin{equation} \label{eq:3.17}
\|(\tilde{\nc}(\th) - \tilde{T}_N(\th)) \Op_h(\chi^{-}_{\delta}) \|_{L^2 \to H_h^{1}} \leq B_N h^{N}, \: \th \in L.
\end{equation}  
The advantage of the above construction is that the symbol of $\tilde{T}_N(\th)$ is holomorphic for $\th \in L.$  
    
 \section{Location of the eigenvalues and incoming resonances}

 Recall that an eigenfunction $f$ of $G$ with eigenvalue $\lambda$ satisfies the equation $\cc(\lambda) f = 0$. The same is true for the incoming resonances $\lambda$. We will use the parametrisation  $\lambda = \frac{\ii \sqrt{z}}{h}, \: z = 1 + \ii \Im z, \: |\im z | \leq 1$ introduced in Section 2 and the equation (\ref{eq:2.13}).

 Denote by $(., .)$ the scalar product in $L^2(\Gamma)$ and by $\|.\|$ the $L^2(\Gamma)$ norm.
  Throughout this section we choose $0 < \delta \ll c_0^2$ and impose the condition 
  \begin{equation}\label{eq:4.1}
-\delta_0 \leq - c h |\log h| \leq  \im z \leq \min\Bigl\{\delta, \frac{\sqrt{1 - c_1^2}}{2 c_1} \sqrt{\delta}\Bigr\}= \delta_0.
 \end{equation}  
 Hear $0 < c_0 \leq c_1 <  1$ are the constants introduced in Section 1 and $\delta$ was used in the construction of the parametrix in the hyperbolic region in Section 3. In subsection 4.1 there are no restrictions on $\delta > 0$, in subsection 4.2 we take $\delta \ll c_0^2$ and in subsection 4.3 we choose $\delta \leq  c_0.$ Notice that for the analysis of eigenvalues we work with  $\im z > 0,$ while for the analysis of incoming resonance we deal with $\im z < 0.$
  
  In \cite{P1} it was proved that for $0 < \gamma(x) < 1$ the eigenvalues $\lambda$ of $G$ are located in the region
 $$\{\lambda \in \C:\: |\re\lambda| \leq C_{\ep}( 1 + |\im \lambda|)^{1/2 + \ep},\: \re \lambda< 0\}, \: 0 < \ep \ll 1.$$
 Since $\re \lambda = - \frac{\im \sqrt{z}}{h}  < 0 ,\: \im \lambda = \frac{\re \sqrt{z}}{h},\: 1 \leq \re \sqrt{z} \leq \sqrt{2},$ for fixed $ \ep < 1/6$ and small $h$ we deduce
 $$  0 <\im z \leq C_0 h^{1/2 - \ep} \leq h^{1/3}, \: \frac{1}{h} \leq |\im \lambda| \leq \frac{\sqrt{2}}{h}$$

 We choose $0 < h \leq \min\{\delta^6, \delta_0\}.$
For fixed $\delta$  if $h$ is small enough,  $0 < \im z \leq h^{1/3}$ implies the inequality on the right hand side of (\ref{eq:4.1}). (In Section 5 we take $h = o((1- c_1)^2)$ in the proof of Proposition 5.1, so for $c_1 \nearrow 1$ we must work with $h$ small enough.) On the other hand, $ - ch |\log h| \leq \im z < 0$ yields
$$0 < \re \lambda \leq \frac{c |\log h|}{2 \re \sqrt{z}}\leq \frac{c}{2} \log |\im \lambda|.$$
Our purpose is to establish that  following 
\begin{thm} Assume the inequalities $(\ref{eq:4.1})$ with $\delta \leq c_0^2$  and $h < \min\{\delta^6, \delta_0\}$  small. Then for every eigenvalues having the form $\ii  \sqrt{z}/h$ and every incoming resonances lying in $M_c$ and having the form $\ii  \sqrt{z}/h$ we have the estimate
\begin{equation*}
0 < |\im z| \leq B_N h^N, \: \forall N \in \N.
\end{equation*}

\end{thm}
Theorem 4.1 implies the statement of Theorem 1.1. Indeed,
  one obtains
   $$ |\im \sqrt{z}| = \big| \frac {\im z}{2 \re \sqrt{z}}\big| \leq \frac{B_N}{2} h^N$$
     and for $0 < h  \leq \delta_0$ we deduce
     $$|\re \lambda | = \Big |\frac{ \im \sqrt{z}}{h}\Big | \leq \frac{B_{N + 1} }{2} h^{N}\leq C_N |\im \lambda|^{-N},\: |\im \lambda| \geq \frac{1}{\delta_0}.$$  
   Next 
   $$\frac{ \sqrt{\delta}}{2 c_1} \leq  \frac{c_0}{2c_1} \leq \frac{1}{2}$$
   and from (\ref{eq:4.1}) we give $\delta_0 \leq  \min\{c_0^2,\frac{\sqrt{1 - c_1^2}}{2}\}.$

   In this section the coordinates  in $T^*(\Gamma)$ are denoted by $(x', \xi').$ The principal symbol of a parametrix for $\nc(z, h) - \gamma\sqrt{z}$ in hyperbolic and elliptic regions (see subsection 3.1 and Section 5 in \cite{V3}) becomes
 \begin{equation} \label{eq:4.2}
 \sqrt{z- r_0} - \gamma \sqrt{z} = \rho - \gamma \sqrt{z}, 
 \end{equation} 
 where
  $$\rho(x', \xi') : = \sqrt{1 + \ii \im z - r_0(x', \xi')}.$$ 
  Let
 $$\Sigma : = \{(x', \xi') \in T^*(\Gamma): \: r_0(x', \xi')  = 1 - \gamma^2(x')\} \subset {\mathscr H},$$
  be the singular set, where the operator $\Op_h(\sqrt{1 - r_0}- \gamma)$ is not elliptic. Then 
  $1 - \gamma^2 \leq 1 - c_0^2  < 1 - \delta$ implies
 $\Sigma \subset \{(x', \xi'): \: r_0 \leq 1 - \delta\}.$
 
  Introduce a partition of unity on $\R$ given by
 $$\psi^{-}(t) + \psi^0(t) + \psi^{+}(t) \equiv 1,\: \forall t \in \R,$$
 where  $\psi^0(t) \in C_0^{\infty} (\R; [0, 1]), \:\psi^0(t) = 1$ for $|t| \leq 1,\: \psi^0(t) = 0$ for $|t| \geq 3/2,$ while
 $$ \psi^{\pm}(t) \in C^{\infty}(\R; [0, 1]),\: \supp \psi^{-} \subset (-\infty, -1),\: \supp \psi^{+} \subset (1, +\infty).$$
  Define the symbols
 $$\chi^{\pm}_{\delta}(x', \xi') = \psi^{\pm}\Bigl(\frac{r_0(x', \xi')- 1}{\delta}\Bigr),\: \chi^0_{\delta}(x', \xi') = \psi^0\Bigl(\frac{r_0(x', \xi')- 1}{\delta}\Bigr)$$
 and notice that $\Sigma \cap \Bigl(\supp \chi^{0}_{\delta} \cup \supp \chi^{+}_{\delta} \Bigr) = \emptyset.$
 Let $Q^{\pm}_{\delta} = \Op_h(\chi^{\pm}_{\delta}),\: Q^0_{\delta} = \Op_h(\chi^0_{\delta})$ be $h$-pseudo-differential operators. For $h^{1/2} \leq \delta$ one obtains the estimates
 $$\|Q^{j}_{\delta}\|_{L^2(\Gamma) \to L^2(\Gamma)} \leq Q_j,\: j = \pm, 0$$
 with constants $Q_j$ independent of $h$ and $\delta$. In fact, 
  $$\sum_{|\alpha| \leq d} h^{|\alpha|/2} \sup_{(x', \xi') \in T^*(\Gamma)}\Bigl |\partial_{x'}^{\alpha} \psi^{j}\Bigl(\frac{r_0- 1}{\delta}\Bigr)\Bigr  | \leq Q_j \: j = \pm, 0$$ 
  and an application of Proposition 3.1 in \cite{V1}  implies result.  
  \subsection{Analysis of $Q^{+}_{\delta}f$}
Here we treat the elliptic region ${\mathscr E}$, where we have a parametrix $S(z, h)$ with principal symbol (\ref{eq:4.2}),
   such that for small $h$ one has
   \begin{equation} \label{eq:4.3}
  \|(\nc(z, h) -S(z, h) )Q^{+}_{\delta} \|  \leq A_1h
  \end{equation} 
  with  constant $A_1 > 0$ independent of $h$. This implies
  $$  |((\nc(z, h) - S(z, h))Q ^{+}_{\delta}f, Q^{+}_{\delta} f )|  \leq A_2h \|Q^{+}_{\delta} f\|^2.$$  
  Here and below we denote by $A_j$ different constants independent of $h$. These constants may depend of $\delta$, while  by $B_j$ we denote different constants independent of $\delta$ and $h$.
 We will prove the following
 \begin{prop} Assume the inequalities $(\ref{eq:4.1})$ and $h$ sufficiently small.Then $Q^+_{\delta} f = 0.$ 
 \end{prop}
 \begin{proof}
 Choose a function  $\zeta(t) = (1 + \delta - t) \alpha(t)$ with $\alpha(t) \in C^{\infty}(\R; [0, 1]) ,\: \alpha(t) = 1$ for $t \leq 1 + \delta,\: \alpha(t) = 0$ for $t \geq 1 + \frac{3\delta}{2}$ and introduce the symbol $r_e = r_{0} + \zeta(r_{0})$. Obviously, $r_e = r_0$ for $r_0 \geq 1 + \frac{3}{2}\delta$ and 
   $$r_e = r_0(1 - \alpha(r_0)) + (1 + \delta) \alpha(r_0) \geq 1 + \delta,\: \forall (x', \xi'). $$
   Therefore
 $$R_e = \sqrt{1 - r_e +\ii t \im z }- \sqrt{1 - r_0 +\ii t \im z} = \frac{r_{0}(1 - \zeta(r_{0}))}{\sqrt{1 - r_e +\ii t \im z} +\sqrt{1 - r_0 +\ii t \im z}}$$
 vanishes on the support of $\chi^{+}_{\f3}$ and $\Op_h(R_e)Q^{+}_{\f3}f$ can be estimated by ${\mathcal O}(h^{\infty}) \|f\|.$ 
  
      Let $ s_e : = \sqrt{1 - r_e + \ii \im z}.$   Then
  $|\im \sqrt{z}| =   \frac{|\im z|}{2\re \sqrt{z}} \leq \frac{|\im z|}{2}.$
On the other hand, 
  $$\im s_e = \re \sqrt{r_e - 1 - \ii \im z} \geq \sqrt{r_e - 1}.$$
According to (\ref{eq:4.1}),
 $$|\im z|  < \frac{\sqrt{\delta}} {c_1} \leq \frac{\sqrt{r_e - 1}}{c_1},$$
  and we give
 $$\im \Bigl(s_e -  \gamma \sqrt{z}\Bigr) \geq  \sqrt{r_e -1} - c_1 \frac{|\im z|} {2} \geq \frac{\sqrt{r_e - 1}}{2} \geq c_2 \sqrt{1 + r_0},$$
 where  $0 <  c_2 \leq \frac{\sqrt{\delta}}{2 \sqrt{2}}$.  In fact, for $r_0 \leq 1 + \frac{3}{2} \delta$ we arrange
 $\delta \geq 4c_2^2( 2 +  \frac{3 \delta}{2})$
 taking 
 $$0 < c_2 \leq \frac{\sqrt{\delta}}{2 \sqrt{2 + \frac{3 \delta}{2}}}\leq \frac{\sqrt{\delta}}{2 \sqrt{2}}.$$
  For $r_0 \geq 1 + \frac{3}{2} \delta$ we obtain $r_e= r_0$ and
 $$4c_2^2 + 1  \leq (1 - 4 c_2^2) ( 1 + \frac{3 \delta}{2})  \leq ( 1- 4 c_2^2) r_0$$
 which is satisfied for $\frac{3}{2} \delta \geq 4c_2^2 (2 + \frac{3 \delta}{2}).$ 
  By a standard argument, we deduce
  \begin{equation} \label{eq:4.4}
 \im \Bigl ( (S(z, h) - \gamma \sqrt{z}) Q^{+}_{\delta} f, Q^{+}_{\delta} f \Bigr) \geq c_2 (\Op_h(\sqrt{1 + r_0}) Q^{+}_{\delta} f, Q^{+}_{\delta} f ) - A_3 h \|\qp f\|^2.
 \end{equation}
 Consequently,
 $$\|(\nc(z, h) - \gamma \sqrt{z} )\qp f\|_{H^{-1/2}} \|\qp f\|_{H^{1/2}}  \geq c_2 \|\qp f\|_{H^{1/2} } - A_4 h \|\qp f\|_{H^{1/2}}^2$$
 and for small $h$ one obtains
 \begin{equation} \label{eq:4.5} 
 \| \qp f\|_{H^{1/2}}  \leq C_2\big \|\Bigl(\frac{(\nc(z, h)}{\sqrt{z}} - \gamma\Bigr)\qp f\big\|_{H^{-1/2}}
 \end{equation} 
 with constant $C_2 > 0$ depending of $\delta.$ The estimate (\ref{eq:4.5}) holds for every function $f \in H^{1/2}$ and not only for eigenfunctions. In Section 6 we will use it  for the proof of (\ref{eq:6.5}).

  As an application of (\ref{eq:4.4}) we will show that $\cc(f) = 0$ implies $\qp f = 0.$ Assume $\| f \| = 1$ and observe that
  $$((S(z, h) - \sqrt{z} \gamma) \qp f, \qp f) = (\qp (S(z, h) - \sqrt{z} \gamma) f, \qp f) + ([S(z, h), \qp] f, \qp f) $$
  $$- ([\sqrt{z} \gamma, \qp] f, \qp f) =\Bigl(([S(z, h), \qp] - [\sqrt{z} \gamma, \qp]) f, \qp f\Bigr)+ \co_{\delta}(h) \|\qp f\|^2.$$ 
  Here in the second equality we used the equality $(\nc(z, h) - \sqrt{z} \gamma) f = 0.$
  On the other hand,
   the symbol of the commutator $[ \Op_h(s_e), \qp  ]$ modulo a remainder $\tilde{b}_{N, e, \delta}$ has the form
  $$b_{N, e, \delta}= \sum_{1 \leq j \leq N-1} \frac{(\ii h)^j}{j!}\sum_{|\alpha|= j}\Bigl[ D_{\xi'}^{\alpha} (s_2 )D_{x'}^{\alpha} (\chi^{+}_{\delta})
  - D_{\xi'}^{\alpha} (\chi^{+}_{\delta}) D_{x}^{\alpha} (s_e) \Bigr]= \sum_{j = 1}^{N-1} b_j$$  
  with $D_{x'} = -\ii \partial_{x'}, \: D_{\xi'} = -\ii \partial _{\xi'}.$ The derivatives  $D_{x'}^{\alpha}(\chi^{+}_{\delta})$ yield a factor $ \delta^{-|\alpha|}$, while the derivatives $D_{\xi'}^{\alpha} s_e$ yield a factor $(1 - r_e + \ii \im z)^{1/2 -|\alpha|}.$ Since $r_e \geq 1 + \f3$,  by using  the condition $h^{1/3} \leq \delta,$ we estimate 
  $$|b_j| \leq \frac{C_j}{j !} h^{j/3}\Bigl(\frac{h^{1/3}}{\delta}\Bigr)^{2j}\sqrt{\delta} \leq D_j \sqrt{\delta} h^{j/3}, \: \forall (x', \xi'),\: 1 \leq j \leq N-1$$
 with constant $C_j$ depending of the derivatives of $\psi^{+}$ and $r_0$ and independent of $\delta$ and $h$. A similar estimate holds for $h^{|\beta|/2}|\partial_{x'}^{\beta} b_j|, \: |\beta| \leq d,$ and for the remainder $\tilde{b}_{N, e, \delta}$. Consequently, 
 \begin{equation} \label{eq:4.6} 
 \|[ S(z, h), \qp ]\|_{L^2 \to L^2} \leq B_5 \sqrt{\delta} h^{1/3}
 \end{equation} 
 with a constant $B_5> 0$ independent of $h$ and $\delta.$ The same argument can be applied to estimate the commutator $[ \gamma, \qp]$. Thus we obtain a upper bound
 $$\big|\im \Bigl ( (S(z, h) - \gamma \sqrt{z}) Q^{+}_{\delta} f, Q^{+}_{\delta} f \Bigr) \big|\leq (B_6 \sqrt{\delta} h^{1/3} + A_5 h) \|\qp f\|^2$$
 with constant $B_6 >0$ independent of $\delta$ and $A_5 > 0$ depending on $\delta$.  Combing this with (\ref{eq:4.4}), taking $c_2 = \frac{\sqrt{\delta}}{2 \sqrt{2}}$ and using $h^{1/3} \leq \delta$, we deduce
 $$\frac{\sqrt{\delta}}{2 \sqrt{2}} \|\qp f\|^2 \leq \Bigl(B_6 \sqrt{\delta} h^{1/3} + (A_3 + A_5)h^{2/3} \delta\Bigr) \|\qp f\|^2.$$
 We fix $\delta$ and the constants $A_3, A_5.$ Choosing $h$ small enough, we obtain $\qp f = 0.$ \end{proof}
 \begin{rem} In this subsection the argument works without restrictions on $0 < \delta < 1.$
 \end{rem}

   \subsection{Analysis of $Q^{0}_{\delta} f$} 
Our purpose is to prove the following
\begin{prop}Assume the inequalities $(\ref{eq:4.1})$ and $h$ small enough.There exists constant $D > 0$ such that for $\delta \leq \frac{c_0^2}{D}$ we have $Q^0_{\delta} f = 0.$ 
\end{prop}
 \begin{proof} First, using $Q_{\f3}^{+} f = 0$, from $\cc(f) = 0$ we deduce
 $$(\nc(z, h) - \sqrt{z} \gamma) (Q_{\f3}^{-} f + Q_{\f3}^{0}f) = 0.$$
 Consider
 $$ \re \Bigl(Q^{0}_{\delta}\Bigl(-\nc(z, h) + \sqrt{z} \gamma \Bigr) (Q^{-}_{\f3}f + Q^{0}_{\f3}f), Q^{0}_{\delta} f \Bigr) = 0.$$  
 It is easy to estimate the terms involving $Q^{-}_{\f3},$    since 
 \begin{equation} \label{eq:4.7}
  \rm{supp} \:\chi^{0}_{\delta} \cap \supp \chi^{-}_{\f3} = \emptyset.
  \end{equation} 
  In the hyperbolic region we have a paramterix $T_N(z, h)$ and we can apply (\ref{eq:3.13}).
  This implies 
  $$|(Q_{\delta}^{0}(- \nc(z, h) + \sqrt{z} \gamma)Q_{\f3}^{-} f, Q_{\delta}^0 f) | \leq C_N h^N \|Q_{\delta}^0 f\|^2$$ 
and we give
  $$ \re \Bigl(Q_{\delta}^{0}\Bigl(- \nc(z, h) + \gamma\sqrt{z}\Bigr) Q^{0}_{\f3}f , Q^{0}_{\delta} f \Bigr) \leq C_N h^N \|Q_{\delta}^0 f\|^2.$$
   Next
  $$Q^{0}_{\delta} \gamma Q^{0}_{\f3}= [Q^{0}_{\delta}, \gamma] Q^{0}_{\f3} + \gamma Q^{0}_{\delta}(Q^{0}_{\f3} - 1) + \gamma Q^{0}_{\delta}.$$
  The commutator on the right hand side has a norm ${\mathcal O}(\frac{h}{\delta})= \co (\delta^2)$ and for the second term we use the fact that $\chi^{0}_{\f3} =1$ on the support of $\chi^{0}_{\delta}.$ Consequently,
  \begin{eqnarray} \label{eq:4.8}
  c_0 \|Q_{\delta}^0 f\|^2 \leq \re \Bigl( \gamma\sqrt{z} Q^0_{\delta}f , Q^{0}_{\delta} f \Bigr) \leq  \re \Bigl(Q_{\delta}^{0}\nc(z, h) Q^0_{\f3}f, Q_{\delta}^0 f\Bigr) \nonumber \\
  +  (C'_N h^N + B_7 \delta^2) \|Q_{\delta}^0 f\|^2
  \end{eqnarray}
  with constant $B_7 > 0$ independent of $\delta$ and $h$.
  
  The problem is reduced to an estimate of the term involving $\nc(z, h) Q_{\f3}^0$ and we prove the following
  \begin{lem} For $\im z$ satisfying $(\ref{eq:4.1})$ with $\delta \leq \frac{c_0^2}{D}$ and $h$ small enough we have
  \begin{equation} \label{eq:4.9}
  \|\nc(z, h) Q_{\f3}^0 \|_{L^2 \to L^2} \leq B_8 \sqrt{\delta} + A_9 h^{1/12}
  \end{equation}
  with $B_8 > 0$ independent of $\delta$ and $A_9 > 0$ independent of $h$.
  \end{lem} 
  \begin{proof} We consider several cases concerning $\im z.$
  
  {\bf 1.} $\im z \geq h^{1/3}.$ Following the results in \cite{V1}, \cite{V2}, \cite{V4}, for $\im z \geq h^{1/2 -\ep}$ and $\ep = 1/6$, we obtain
  $$\|(\nc(z, h) - \Op_h(\rho))Q_{\f3}^0\|_{L^2 \to L^2} \leq C_1 \frac{h}{\im z}  \leq C_1 h^{2/3}$$
  with constant $C_1 > 0$ independent of $h$ and depending on $\delta.$ For $\im z \leq \delta$ on the support of $\chi_{\f3}^0$  we have
  $|\sqrt{1 - r_0 + \ii \Im z}| \leq \sqrt{\frac {13 \delta}{4}}$. Applying Proposition 3.1 in \cite{V1} we deduce
  $$\|Op_h(\rho) Q_{\f3}^0\|_{L^2 \to L^2} \leq B_9 \sqrt{\delta}.$$
  
  {\bf 2.} $h^{2/3}\leq \im z \leq h^{1/3}.$      
  In this case, we will apply Proposition 3.3 in \cite{V3}. Set
 $\tl = -\ii \lambda = \frac{\sqrt{z}}{h}$ and for $f \in L^2(\Gamma)$ consider the problem
 \begin{equation} \label{eq:4.10}
 \begin{cases} (\Delta + \tl^2) u = 0 \: \rm {in}\: \Omega,\\
 u  = Q^{0}_{\delta} f\:\rm{on} \: \Gamma,\\
 u -(\ii \tl)-\rm {incoming}.\end{cases}
 \end{equation}
 Here we replace $\delta^2$ in Section 3, \cite{V3} by $\delta.$ Since $\im z > 0$ and $u$ is $\ii \sqrt{z}/h$-incoming, $u$ is decreasing for $|x| = r \to \infty.$  Hence we may use an  integration by parts to obtain Lemma 3.1 in \cite {V3} for unbounded domains and the proof of Proposition 3.3 in \cite{V3} works.  To apply this proposition, we need the conditions 
 $$1 \leq |\im \tl | \leq \delta \re \tl, \: \re\tl \geq C_{\delta} \gg 1.$$
 Obviously, 
 $$\im \tl = h^{-1} \frac{\im z}{2 \re \sqrt{z}},\: \re \tl = h^{-1} \re \sqrt{z} \geq h^{-1}.$$
  The analysis of the proof of Proposition 3.3 in  \cite{V3} shows that it is sufficient to take $C_{\delta} = \frac{1}{\delta^{1/8}}$ and  $\frac{\re\sqrt{z}}{h} \geq \frac{1}{h} \geq \frac{1}{\delta^{1/8}}$   holds. On the other hand, for $ \im z \geq h^{2/3}$ and small $h$ one gets
 $$1 \leq \frac{1}{2h^{1/3} \sqrt{2}} \leq \frac{ \im z}{2 h \re \sqrt{z}}$$
 and $\Im z \leq 2 \delta (\re \sqrt{z})^2$ is satisfied by (\ref{eq:4.1}). 
 Now the estimate (3.11) in \cite{V3} for the solution $u$ of (\ref{eq:4.10}) yields
 $$\|\nc(z, h) Q_{\delta}^0 f\|_{L^2(\Gamma)} \leq B_{10} \Bigl( \sqrt{\delta} + |\im \tl |^{-1/4}\Bigr)\|f\|_{L^2(\Gamma)}\leq B_{10}\Bigl(\sqrt{\delta} + 2^{3/8} h^{1/12}\Bigr)\|f\|_{L^2(\Gamma)}$$
 with a constant $B_{10} > 0$ independent of $\delta$ and $\tl.$ We replace $\delta$ by $3\delta/2$ changing the constant $B_{10}$.
 
{\bf 3.} $- c h |\log h| \leq \im z \leq h^{2/3}$. In this case we have $|\im z| \leq h^{2/3}$ and a semi-classical parametrix for the exterior Dirichlet-to-Neumann map $\nc(z, h)$ for strictly convex obstacle has been constructed in Chapter 10 of \cite{Sj}. In particular, the estimate (10.32) in \cite{Sj} for the principal symbol $n_{ext}(x', \xi'; h)$ of $\nc(z, h)$ in suitable local coordinates yields
$$|\pa_{x'}^{\alpha}\pa_{\xi'}^{\beta} n_{ext}(x', \xi'; h)| \leq C_{\alpha, \beta} (|1- r_0| + h^{2/3})^{1/2 - \beta_1},$$
where $\xi_1$ is the dual variable after a second microlocalisation. The derivatives with respect to $x'$ are estimated by $\co (|1- r_0|^{1/2} + h^{1/3})$ and applying Proposition 3.1 in \cite{V1} once more, we deduce
 \begin{equation*}
  \|\nc(z, h) Q^{0}_{\delta}\|_{L^2 \to L^2}  \leq  B_{11}  (\sqrt{\delta} + h^{1/3})
  \end{equation*} 
  with a constant $B_{11} > 0$ independent of $\delta$ and $h$.   
  Combining the estimates in the three cases we deduce (\ref{eq:4.9}).
  \end{proof}
 
 \begin{rem} The estimate $(\ref{eq:4.9})$ is not optimal. The term $h^{1/12}$ comes form the case $2$, where we have used the results of \cite{V3} established for obstacles with arbitrary geometry. For strictly convex obstacles we may apply the parametrix constructed in Section 5, \cite{P1} for $h^{2/3} \leq \im z \leq h^{\ep}, \: 0 < \ep \ll 1.$ However, Theorem $5.2$ in \cite{P1} yields an estimate for $\|\nc(z, h) \Op_h(\chi_{\ep/2}^0)\| = \co(h^{\ep/4})$ and some extra work is necessary. Moreover,  if $h^{\ep/2} \leq \delta $, then $h$ should be taken very small for $\ep \ll 1.$
\end{rem}
Going back to (\ref{eq:4.8}), we obtain the estimate
$$c_0 \|Q_{\delta}^0 f\|^2 \leq (B_{12} \sqrt{\delta} + A_{10} h^{1/12} ) \|Q_{\delta}^0 f\|^2$$
with a constant $B_{12} > 0$ independent of $\delta.$ We fix $\delta = \frac{c_0^2}{4 B_{12}^2} $ and then $A_{10}$ will be fixed too. Finally, for small $h$ we arrange $A_{10} h^{1/12} < c_0 /2$ and conclude that $Q_{\delta}^0 f = 0.$ \end{proof}

 For further references notice that for all $ f \in H^{1/2}$,  $\delta \leq \frac{c_0^2}{D}$ and small $h$ we have an estimate similar to (\ref{eq:4.5})
  \begin{equation} \label{eq:4.11}
  \|\qn f\|_{H^{1/2}} \leq C_4 \big\|\Bigl(\frac{\nc(z, h)}{\sqrt{z}} - \gamma \Bigl) \qn f\big \|_{H^{-1/2}}
  \end{equation} 
  with constant $C_4 > 0$ depending of $\delta$. To do this, consider
  $$\re\Bigl((- \frac{(\nc(z, h)}{\sqrt{z}} + \gamma )\qn f, \qn f\Bigr)$$  
  and apply (\ref{eq:4.9}) for the norm of $\nc(z, h).$
  
 \subsection{Analysis of $Q^{-}_{\delta} f$}
 In this subsection we prove Theorem 4.1.\\
 {\it Proof of Theorem 4.1}
 Set $\fh := Q^{-}_{\delta} f$ and consider 
 $$0 = \im \Bigl( (\nc(z, h) - \gamma \sqrt{z} ) \fh, \fh \Bigr)$$
 $$ = \im \Bigl( (T_{N}(z, h) - \gamma \sqrt{z} )\fh, \fh \Bigr) + \co (h^N)\|\fh \|^2,$$
where $T_N(z, h)$ is the parametrix in the hyperbolic region constructed in subsection 3.1.
 The operator $T_N(1, h)$ is self-adjoint, hence $\im (T_N(1, h) \fh, \fh) = 0.$ This implies
  \begin{eqnarray} \label{eq:4.12}
 \Bigl |\im \Bigl( (T_{N}(z, h) - \gamma \sqrt{z} )\fh, \fh \Bigr)\Bigr | \nonumber\\
 =  |\im z| \Bigl | \re\Bigl( \Bigl(\frac{\partial T_{N}(z_{t}, h)}{\partial z} -  \frac{\gamma}{2\sqrt{z_{t}}}\Bigr) \fh, \fh\Bigr )\Bigr | = \co (h^{N})\|\fh \|^2
 \end{eqnarray} 
 with $z_{t}= 1 + \ii t \im z,\: 0 < t < 1,\:|\im z| \leq \delta $. Introduce the operator
 $$F = \Op_{h}\Bigl(\frac{1}{ 2\sqrt{z_{t} - r_{0}}}\Bigl) - \frac{\gamma}{2 \sqrt{z_{t}}}$$
 As in  Lemma 3.9 in \cite{V1} and Lemma 4.1 in \cite{P1} we obtain 
 \begin{equation} \label{eq:4.13}
 \Bigl |\re\Bigl( \Bigl(\frac{\partial T_{N}(z_{t}, h)}{\partial z} - \frac{\gamma}{ 2\sqrt{z_{t}}}\Bigr) \fh, \fh \Bigr) - \re (F (\fh), \fh)\Bigr| \leq A_{11} h \|\fh\|^2
 \end{equation}
 with constant $A_{11}$ independent of $h.$
 Next $\re (F g, g) = \frac{1}{2} ((F + F^*)g, g)$ and
 the principal symbol of the self-adjoint operator $\frac{1}{2} (F + F^*)$ is
 $$\frac{1}{2}\re\Bigl(\frac{1}{\sqrt{1 +\ii t \im z - r_{0}}} - \frac{\gamma}{\sqrt{1 + \ii t \im z}}\Bigr).$$
 Let $\beta(t) \in C^{{\infty}}(\R; [0, 1])$ be such that $\beta(t) = 0$ for $t \leq 1- \delta$ and $\beta(t) = 1$ for $t \geq 1- \delta/2,\: \beta'(t) \geq 0,\: \forall t \in \R.$ Introduce the symbol 
   $$\tr( x', \xi'):  = r_{0}(1 - \beta(r_0 ))+ \Bigl(1 -\frac{\delta}{2}\Bigr)\beta(r_0)$$
   and notice that
   $$1 -\tr = (1 - r_0)(1 - \beta(r_0)) + \frac{\delta}{2} \beta(r_0).$$
 Then $r_0 = \tr$ on $\supp \chi^{-}_{\delta}$ and for $1 - \delta \leq r_0 \leq 1- \delta/2$ one deduces
$$1 - \tr \geq \frac{\delta}{2} (1 - \beta(r_0)) + \frac{\delta}{2} \beta(r_0) = \delta/2,$$
because $1- \frac{3\delta}{2}  \geq 1- c_0^2 >0.$ Consequently, 
  $$1- \tr \geq \delta/2, \: \forall (x', \xi') \in T^*(\Gamma).$$  
 Moreover,
 $$R_{h} = \sqrt{1 +\ii t \im z - \tr}- \sqrt{1 +\ii t \im z - r_{0}} = \frac{r_{0}- \tr}{\sqrt{1 +\ii t \im z - \tr} +\sqrt{1 +\ii t \im z - r_{0}}}$$
 implies $\|\Op_h(R_h)\fh\| = {\mathcal O}(h^{\infty})\|\fh\|.$
  The same is true for the operator with symbol
 $$ \frac{1}{\sqrt{1 +\ii t \im z - \tr}}- \frac{1}{\sqrt{1 +\ii t \im z - r_{0}}}.$$ 
  We are going to study the operator with  principal symbol
 \begin{equation} \label{eq:4.14}
 s(x', \xi'; z) = \re\Bigl(\frac{1}{\sqrt{1 +\ii t \im z - \tr}} - \frac{\gamma}{\sqrt{1 + \ii t \im z}}\Bigr).
 \end{equation}
 Our purpose is to prove that for $\im z$ satisfying (\ref{eq:4.1}), the symbol $s(x', \xi'; z)$ is elliptic.
 
 Set $y = \sqrt{1 + t^2 (\im z)^2},\: q = \sqrt{(1 - \tr)^2 + t^2 (\im z)^2},$ and let
 $$z_{t} = ye^{\ii \varphi},\: 1- \tr + \ii \Im z = q e^{\ii \psi},\: 0 \leq |\varphi| \leq \pi/4,\: 0 \leq |\psi| \leq \pi/4.$$
 Therefore
 $$s(x', \xi'; z)= \re \Bigl (\frac{e^{-\ii \psi/2}}{\sqrt{q}}- \gamma \frac{e^{-\ii \varphi/2}}{\sqrt{y}}\Bigr)= \frac{1}{\sqrt{y q} }\Bigl[ \sqrt{y} \cos \frac{\psi}{2} - \gamma \sqrt{q} \cos \frac{\varphi}{2} \Bigr]$$
$$= \frac{1}{\sqrt{y q} }\frac{ y \cos^2 \frac{\psi}{2} - \gamma^2 q \cos^2 \frac{\varphi}{2}} {\sqrt{y} \cos \frac{\psi}{2} + \sqrt{q} \cos \frac{\varphi}{2}}$$
$$= \frac{1}{2\sqrt{y q}}\frac{ y (1 + \cos \psi) - \gamma^2 q (1 + \cos \varphi)} {\sqrt{y} \cos \frac{\psi}{2} + \sqrt{q} \cos \frac{\varphi}{2}}. $$
 On the other hand,
 $$ y (1 + \cos \psi) - \gamma^2 q (1 + \cos \varphi) = y\Bigl (1 + \frac{1 - \tr}{q}\Bigr) - \gamma^2 q \Bigl(1 + \frac{1}{y}\Bigr)$$
  $$= \frac{1}{ y q} \Bigl[ y^2(1 - \tr + q) - \gamma^2 q^2(1 + y)\Bigr].$$
  Write  the symbol in the brackets $[ ... ]$ as 
  $$s_{1}= y^2 \Bigl( 2(1- \tr) + \frac{  t^2 (\im z)^2}{q + 1 - \tr}\Bigr) - \gamma^2( 1 + y) ( (1- \tr)^2 + t^2 (\im z)^2)$$
  $$\geq (1 + y) \Bigl[(1- \tr)(y - \gamma^2(1 - \tr)) - \gamma^2 t^2 (\im z)^2 + \frac{y^2 t^2 (\im z)^2}{(1+y)(q + 1 - \tr)}\Bigr]$$
  $$\geq (1 + y)\Bigl[ d_{0} - \gamma^2 (\im z)^2\Bigr]$$
with 
$$d_{0} = (1 - \tr)(1 - \gamma^2(1 - \tr)) \geq (1- \tr)( 1- c_1^2 (1- \tr)) = G(\tr).$$
The function $G(\tr)$ has a maximum for $\tr  = 1 - \frac{1}{2 c_1^2} = c_3$ and $c_3 < 1- \delta/2.$
  Let $\min_{\tr \in [0, 1- \delta/2]} G(\tr)= d_1.$ If $c_1^2 \leq 1/2$, then $d_1 = \frac{\delta}{2} (1- c_1^2 \frac{\delta}{2}).$ For $0 < c_3 < 1- \delta/2$, we have 
  $$d_1 = \begin{cases} \frac{\delta}{2}( 1- c_1^2 \frac{\delta}{2}),\:{\rm if}\: \frac{1}{2} < c_1^2 \leq \frac{1}{1 + \delta/2},\\
  1- c_1^2,\: {\rm if}\: c_1^2 > \frac{1}{1 + \delta/2}. \end{cases}$$
For simplicity we use the crude estimate $d_0 \geq \frac{\delta}{2}( 1- c_1^2)$.

By (\ref{eq:4.1})
\begin{equation} \label{eq:4.15}
|\im z| \leq \frac{\sqrt{\delta}}{2 c_1 }\sqrt{ 1 - c_1^2} \leq \frac{1}{\gamma} \sqrt{\frac{d_0}{2}}
\end{equation}
 and we obtain
$$s(x', \xi'; z) \geq \frac{d_{0}}{4(yq)^{3/2}(\sqrt{y} + \sqrt{q})} \geq \frac{d_0}{4y^2q^{3/2}} \geq \frac{d_0}{8(1+ \delta^2)^{3/4}}.$$
 Notice that  for some values of $\gamma(x)$ the term
   $$Y= - \gamma^2   + \frac{y^2 }{(1+y)(q + 1 - \tr)}$$
   could be positive. On the other hand, 
   $$\frac{y^2 }{(1+y)(q + 1 - \tr)}   = \frac{1}{4(1-\tr) }+ {\mathcal O}(|\im z|)$$
   and for $\tr = 0, \: \im z = 0,$ the function $Y $ may take negative values. For this reason we simply estimate $Y$  from below by $-\gamma^2$.
   Finally,
   $$\re (F(\fh), \fh) \geq \Bigl(\frac{d_{0}}{8( 1+ \delta^2)^{3/4}}- A_{12} h\Bigr) \|\fh\|^2$$
   with $A_{12} > 0$ independent of $h$.
   Assuming $\fh \neq 0$ and taking together (\ref{eq:4.12}) and (\ref{eq:4.13}), for small $h$ we have
   $$|\im z| \leq B_{N}h^N,\: \forall N \in \N.$$ 
 This completes the proof of Theorem 4.1.
   
   
 \section{Properties of the operator $P(\th)$ for real $\th$}

In the section we use the notations of preceding sections.  Let $\th \in L, \: h = \re \th$. Introduce a function $\alpha(x', \xi') \in C^{\infty}(T^*(\Gamma):\: [0, 1])$ such that 
$$\alpha(x', \xi')= \begin{cases} 0,\:\: {\rm if}\: r_0  \geq  1-\frac{c_0^2 (1+ \ep_1)}{2},\\
1,\:\:{\rm if}\: r_0  \leq 1 -\frac{3 c_0^2}{5},\end{cases}$$
where $0 < \ep_1 \ll 1$ satisfies the inequality $5(1 +\ep_1) < 6.$ Let $\tilde{T}_N(\th)$  be the parametrix  in the hyperbolic region ${\mathscr H}$ constructed in subsection 3.2 for the problem (\ref{eq:3.14}) with boundary data $\Op_{\th}(\alpha) f.$ Recall that this parametrix is holomorphic for $\th \in L.$
 As in (\ref{eq:3.17}) we obtain
$$(\tilde{\nc}(\th) - \tilde{T}_N(\th))\Op_{\th}(\alpha)  = {\mathcal R}_N(h)\Op_{\th}(\alpha), \th \in L$$ 
with $\|{\mathcal R}_N\|_{L^2(\Gamma) \to H^1(\Gamma)} = \co(|h|^N).$ 
The operator $-\tilde{T}_N(h)\Op_h(\alpha) $ is self-adjoint. However for $h^2r_0 > 1$ the principal symbol $- \sqrt{ 1- h^2r_0}$ of $\tilde{T}_N(h)$ is not real valued (see the end of Section 3).
 To deal with a holomorphic in $\th$ operator, we add some additional terms. We fix $\delta = \frac{c_0^2}{2}$ and consider the function $\beta(t) \in C^{\infty}(\R; [0, 1])$ introduced in the previous section with this $\delta$. Recall that $\beta(r_0(x', \xi')) = 0$ for $r_0 \leq 1 - \frac{c_0^2}{2}$ and $\beta(r_0(x', \xi')) = 1$ for $r_0 \geq 1 -\frac{c_0^2}{4}.$ This choice of $\delta$ is {\bf not related to that in Section 4.} In fact, in the proof of Theorem 4.1 we have choose $\delta \ll c_0$ and we obtained an eigenvalue free region ${\mathcal Q}_N$ depending of $c_0$ and $c_1$ and small $h$.   The function $\beta$ and $\delta$ will be fixed and there are no confusion. Later in Section 6 we will use the functions $\chi_{\omega}^{\pm}, \chi_{\omega}^0$ with small parameter $\omega$ different from $\delta$ fixed here.

 Set $A(h) = \Op(1 - \beta(r_0(x', h\xi'))$. In the following $\Op(a)$ with symbol $a(x', h\xi'; h)$ denotes the classical pseudo-differential operator
$$({\rm Op}(a)f)(x) = (2\pi)^{d-1} \iint e^{\ii \la y' - x', \xi'\ra} a(x', h \xi'; h) f(y') dy'd\xi'$$ 
 and similarly if one has $\th$ at the place of $h$. Obviously, 
 \begin{equation} \label{eq:5.1} 
  \sup (1- \beta(r_0(x', \xi')))  \cap \supp \alpha(x', \xi') = \emptyset.
 \end{equation}
 Let
$$ Q(\th) = -\tilde{T}_N(\th)A(h) + \gamma(x) + (Id -A(h)) \Op(\sqrt{1 +\th^2 r_0}).$$

  To extend $Q(\th)$ for $\th \in L$ we must extend holomorphically for $\th \in L$ the operator $A(h)$.
  Since the symbol of this operator vanishes for large $\|\xi'\|$, we may apply Proposition A.1 in \cite{SjV}.  According to this result, for small $h_0$ and $\th \in L$ there exists an operator $R(\th)$ such that
 $$\pa_{\th} ^kR(\th) = \co(|\th|^{\infty}):\: H^{-s} \to H^s,\: \forall k \geq 0, \: \forall s \geq 0$$
 and $\tilde{A}_0(h) = A(h)+ R(h)$ extends holomorphically to the domain $L$. Moreover, the extension $\tilde{A}_0(\th)$ has the form  $\tilde{A}_0(\th)= B(h; \frac{\th}{|h|})$ with a classical pseudo-differential operator $B \in L^{0, 0}_{cl}(\Gamma)$ and $\rho \notin \widetilde{WF}(A(h))$ implies $\rho \notin \widetilde{WF} (\tilde{A}_0(\th)).$  After this manipulation we introduce the operator
$$P(\th) : = -\tilde{T}_N(\th) \tilde{A}_0(\th) + \gamma(x) +(Id-  \tilde{A}_0(\th))  \Op\Bigl(\sqrt{1 + \th^2 r_0}\Bigr)$$
 which is holomorphic for $\th \in L$. Notice that (\ref{eq:5.1}) implies
 $$\widetilde{WF}(Id - A(h)) \cap \widetilde{WF}(\Op(\alpha(x', h \xi')) = \emptyset,$$
 hence
 $$\widetilde{WF}(Id - \tilde{A}_0(\th)) \cap \widetilde{WF}(\Op(\alpha(x', h \xi')) = \emptyset.$$
  This leads to
 \begin{eqnarray} \label{eq:5.2}
 (C(\th) - P(\th)) \Op(\alpha)= -\Bigl(\tilde{\nc}(\th) -\tilde{T}_N(\th)\tilde{A}_0(\th)\Bigr ) \Op(\alpha)+
 \co(|h|^{\infty})\nonumber \\
 =  -\Bigl(\tilde{\nc}(\th) -\tilde{T}_N(\th)\Bigr ) \Op(\alpha) + \co(|h|^{\infty} ) = \co(|h|^N). \end{eqnarray}
 In this section we study the self-adjoint operator $P(h)$ for $\re \th = h$
  with principal symbol  
$$p_1(x', h \xi') : = -\sqrt{1 - h^2r_0} ( 1 - \beta(h^2 r_0))  + \gamma(x) +  \beta(h^2r_0) \sqrt{1 + h^2 r_0}.$$

  Choose  $ 0 < 2\eta = 1 - c_1^2 \leq 1 - 2\delta,\: \ep_0 = \frac{\eta^2}{8(1 + \eta).}$ For two operators $L_1, L_2$ we say that
 $L_1 \leq L_2$ if $(L_1 u, u) \leq (L_2 u, u).$ 
\begin {prop}
For small $h$  depending of $\ep_0$ we have the inequality
\begin{equation} \label{eq:5.3}
h \frac{d P(h)}{dh}  +   P(h) \Op((1 + h^2 r_0)^{-1/2}) P(h) \geq \ep_0 \Op(\sqrt{1 + h^2 r_0}).
\end{equation} 
\end{prop}
\begin{proof}
We have
$$h \frac{d p_1}{dh} = h^2 r_0 \Bigl( \frac{1 - \beta (h^2 r_0)}{\sqrt{1 - h^2r_0}} + \frac{\beta(h^2 r_0)}{\sqrt{1 + h^2 r_0}}\Bigr)$$
$$ + 2h^2 r_0 \beta'(h^2 r_0) \Bigl( \sqrt{1 - h^2r_0} +  \sqrt{1 + h^2 r_0}\Bigr)$$
and  one obtains $h \frac{d p_1}{dh} \geq 0, \: \forall (x', \xi').$\\

 We consider several cases.\\
{\bf 1}. $0 \leq h^2 r_0 \leq \eta.$ Then $\sqrt{1 - h^2r_0} \geq \sqrt{1 - \eta},\: \beta(h^2 r_0) = 0$ and
$$p_1 = - \sqrt{1 - h^2 r_0} + \gamma \leq -\sqrt{1 - \eta} + c_1 =-  \frac{\eta}{c_1 + \sqrt{1 - \eta}}.$$
This implies
\begin{equation}\label{eq:5.4}
h \frac{d p_1}{dh} + \frac{p_1^2}{\sqrt{1 + h^2 r_0}} \geq \frac{\eta^2}{(1 + \eta)(c_1 + \sqrt{1 -\eta})^2}\sqrt{1 + h^2 r_0} > 2\ep_0\sqrt{1 + h^2 r_0}.
\end{equation} 
{\bf 2}. $\eta < h^2 r_0 \leq 1 - \delta.$ Then $\beta(h^2 r_0) = 0$ and
\begin{equation} \label{eq:5.5}
h \frac{d p_1}{dh}  + \frac{p_1^2}{\sqrt{1 + h^2 r_0}} \geq \frac{\eta}{\sqrt{1 - h^2 r_0}} \geq \eta \sqrt{1+ h^2 r_0}.
\end{equation} 
{\bf 3}. $1- \delta < h^2 r_0 \leq 1 - \frac{\delta}{2}.$ 
We obtain
\begin{eqnarray} \label{eq:5.6}
h \frac{d p_1}{dh} + \frac{p_1^2}{\sqrt{1 + h^2 r_0}} \geq (1 - \delta)\Bigl( \frac{1 - \beta(h^2 r_0)}{\sqrt{\delta}} + \frac{\beta(h^2 r_0)}{\sqrt{2 - \delta/2}}\Bigr)\nonumber \\
 > \frac{1- \delta}{\sqrt{2 - \delta/2}} \geq \frac{1-\delta}{2- \delta/2}\sqrt{1 + h^2 r_0} > \frac{1 -\delta}{2}\sqrt{1 + h^2r_0}.
\end{eqnarray} 
{\bf 4}. $h^2 r_0 > 1 - \frac{\delta}{2}.$ Therefore, $\beta (h^2 r_0) = 1$ and 
$$ \frac{ p_1^2}{\sqrt{1 + h^2 r_0}} >  \sqrt{1 + h^2 r_0}.$$
Clearly, $\eta < \frac{1 - \delta}{2}$
and $2\ep_0 = \frac{\eta^2}{4(1 + \eta)} < \eta.$
Taking together the estimates (\ref{eq:5.4}), (\ref{eq:5.5}), (\ref{eq:5.6}), we get
    $$h\frac{ d p_1}{dh} + \frac{p_1^2}{\sqrt{1 + h^2 r_0}} - 2\ep_0 \sqrt{1 + h^2 r_0} \geq 0.$$
By the sharp semiclassical G\"arding inequality we deduce
$$\Bigl(h \Op(\frac{d p_1}{dh})f + \Op\Bigl (\frac{p_1^2}{\sqrt{1 + h^2 r_0}}\Bigr)f, f \Bigr)\geq 2\ep_0\Bigl( \Op(\sqrt{1 + h^2 r_0})f, f\Bigr) - C_1 h \|f\|_{H^{1/2}}^2.$$

The norm of the low order terms on the left hand side of (\ref{eq:5.3}) can be estimated by $C_2 h\|f\|^2,\: C_2 > 0$. Finally,
$h\Bigl(C_1 \|f\|^2_{H^{1/2}}+ C_2\|f\|^2 \Bigr)$ can be absorbed by $\ep_0\|f\|_{H^{1/2}}^2$ choosing $h$ small, and this completes the proof.
\end{proof}
\begin{rem} The choice of $\ep_0$ is independent of $\delta$ and the function $\beta(r_0)$ and $\ep_0 = o((1 - c_1)^2),$  so $\ep_0 \searrow 0$ as $c_1 \nearrow 1.$ We have a similar phenomenon in the case  $\gamma(x) > 1, \: \forall x \in \Gamma$ $($see Remark $2$ in \cite{P2}$)$, where $\ep_0 = o((1 - \tilde{c}_0)^2),\: \tilde{c}_0 = \min_{x \in \Gamma} \gamma(x) > 1.$ On the other hand,  $h$ depends of $\ep_0$ and we must take $h \leq h_0= o((1 - c_1)^2).$
\end{rem}

Let $0 < h \leq h_0 \ll 1$ and let
$$\mu_1(h) \leq \mu_2(h)\leq ...\leq \mu_k(h) \leq ...$$
be the eigenvalues of the self-adjoint operator $P(h).$ We repeat without changes the arguments of Section 4 in \cite{P2} and Section 4 in \cite{SjV}. For convenience of the reader, we mention briefly the main steps and we refer to \cite{P2}, \cite{SjV} for more details. The number of negative eigenvalues
$$k_0 = \sharp \{ k: \: \mu_k(h_0) \leq 0\} = (2 \pi h_0)^{-(d-1)} \iint_{p_1(x', \xi') \leq 0} dx'd\xi' + \co(h_0^{-d + 2})$$
is finite and 
\begin{equation} \label{eq:5.7}
{\mathcal E}  = \{ (x', \xi'): \:p_1(x', \xi') \leq 0\}  = \{(x', \xi'):\:  r_0(x', \xi') \leq 1 - \gamma^2(x')\}.
\end{equation}
In fact,
$p_1 \leq 0$ implies  
$$r_0 \leq 1 - \gamma^2 \leq 1- c_0^2.$$
On the other hand, $\beta(r_0) >  0$ yields $r_0 > 1- \delta = 1 - \frac{c_0^2}{2}$ and this
 contradicts the above inequality. Thus $(x', \xi') \in {\mathcal E}\: \iff \beta(r_0(x', \xi')) = 0$ and we obtain (\ref{eq:5.7}).

 Notice that  $k_0$ is independent of $\delta$ and the function $\beta$. For $k > k_0$ the eigenvalues $\mu_k(h_0)$ are positive. By applying 
(\ref{eq:5.3}), we show that $\mu_k(h)$ are locally Lipschitz functions and the derivatives $\frac{d \mu_k(h)}{dh}$ are almost defined. Repeating the argument in \cite{P2}, \cite{SjV},  for $\mu_k(h) \in [-\alpha, \alpha],\: 0 < \alpha \ll 1, h \in [h_1, h_2]$, one establishes the estimates
$$ \frac{\ep_0}{2} \leq h \frac{d \mu_k(h)}{dh} \leq c_2,\: k > k_0.$$

We discuss briefly only the proof of lower bound in  the above estimate.
   Let $h_1$ be small and let $\mu_k(h_1)$ have multiplicity $m$. For $h$ sufficiently close to $h_1$ one has exactly $m$ eigenvalues and we denote by $F(h)$ the space spanned by them. 
We denote by $\dot{a}(h)$ the derivatives of $a(h)$ with respect to $h$. 
  Let $h_2$ be close to $h_1$ and let $e(h_2)$ be a normalised eigenfunction with eigenvalue $\mu_k(h_2)$. We  construct a smooth extension $e(h) \in F(h),\: h \in [h_1, h_2]$ of $e(h_2)$ with $\|e(h)\| = 1,\:\dot{e}(h) \in F(h)^{\perp}$. Obviously, $e(h_1)$ will be normalised eigenfunction with eigenvalue $\mu_k(h_1).$   
  Since $\mu_k(h) \in [-\alpha, \alpha], \: h \in [h_1, h_2],$ we have $\|P(h) e(h)\| \leq \alpha.$    
      To estimate $h \frac{d \mu_k(h)}{d h}$ from below, we apply (\ref{eq:5.3}). For  $ \alpha = \sqrt{\frac{\ep}{2}}$ we have
$$h \frac{d \mu_k(h)}{d h} = ( h \dot{P} (h) e(h), e(h))  \geq \ep_0 (\Op(\sqrt{1 + h^2 r_0})e(h), e(h))  $$
$$- (\Op((1 + h^2 r_0)^{-1/2}) P(h) e(h), P(h) e(h) ) \geq \ep_0- \alpha^2 \geq \ep_0/2.$$
 Consequently, for $h \in [h_1, h_2]$ one has
$$\mu_k(h_2) - \mu_k(h_1) \geq \frac{\ep_0}{2}\int_{h_1}^{h_2} h^{-1}  dh \geq \frac{\ep_0}{2 h_2}   (h_2- h_1)$$
and 
$$\frac{\ep_0}{2} \leq h \frac{d \mu_k(h)}{dh} .$$
The above inequality combined with the continuity of $\mu_k(h)$ implies that  if for $ h < h_0, \: k > k_0$ we have $\mu_k(h) < 0$, then there exists unique $h < h_k < h_0$ such  that $\mu_k(h_k) = 0.$ Clearly, the operator $P(h_k)$ is not invertible. Thus we are led to count for $0 < \frac{1}{r} \leq h_0$ the number of the negative eigenvalues of $P(1/r)$ which can be expressed by well known formula.

Repeating without any change the argument in \cite{SjV}, we choose $p > d$ and construct intervals $I_{k, p}$ containing $h_k$ with length $|I_{k, p}| \sim h^{p +1} $ and $|\mu_k(h)| \geq h^p$ for $h \in (0, h_0] \setminus I_{k, p}.$ Next, one constructs closed disjoint intervals $J_{k, p}, \: |J_{k, p}| = {\mathcal O} (h^{p + 2 - d})$  so that 
$$\bigcup_{k > k_0} I_{k, p} = \bigcup_{k> k_0} J_{k, p}$$ 
and  we obtain the following
 \begin{prop} [Prop. 4.1, \cite{SjV}] Let $p > d$ be fixed. The inverse operator $P(h)^{-1} : L^2 \rightarrow L^2$ exists and has norm ${\mathcal O}(h^{-p})$ for $h \in (0, h_0] \setminus \Omega_p,$ where $\Omega_p = \cup_{k} J_{k, p}.$  Moreover, the number of intervals  $J_{k, p}$ that intersect $[h/2, h]$ for $0 < h \leq h_0$ is at most ${\mathcal O}(h^{1- p}).$ 
\end{prop} 

\section{Relation between the trace formulas for $P(\th)$ and $C(\th)$}

In this section we study $C(\th)$ and $P(\th)$ for complex $\th \in L$.
Moreover, $\delta = \frac{c_0^2}{2}$ and $\beta$ are fixed as in the preceding section. We obtain without changes many statements of Sections 5, 6 in \cite{SjV}. We refer to \cite{SjV} for the details and below precise citations are given.
First, one repeats the proof of Lemma 5.1 in \cite{SjV} exploiting Proposition 5.1. Thus we get
\begin{equation} \label{eq:6.1}
\|P(\th)^{-1} \|_{\lc(H^{-1/2}, H^{1/2})} \leq C \frac{ h}{|\Im \th|}, \: h > 0, \: \Im \th \neq 0.
\end{equation} 
Second, as in Section 5 in \cite{P2}, one introduces an elliptic operator $M(\th)$ holomorphic in $L$ such that
 $P(\th) - M(\th) : {\mathcal O}_s(1): H^{-s} \rightarrow H^s, \: \forall s.$ A modification is necessary since the principal symbol $p_1$ vanishes on the set $\Sigma = \{(x', \xi'):\: p_1(x', \xi') = 0\}$. For this purpose we repeat the argument of Section 5 in \cite{P2} and for convenience of the reader we present the proof.

 Consider a  symbol $\sigma(x', \xi') \in C^{\infty}_0 (T^*(\Gamma); [0, 1] )$ such that
$$\sigma(x', \xi) = \begin{cases} 1,\:\: (x', \xi') \in T^*(\Gamma),\: r_0(x', \xi') \leq 1-  \frac{11c_0^2}{12} ,\\
0,\:\: (x', \xi') \in T^*(\Gamma),\:r_0(x', \xi')  \geq 1- \frac{5c_0^2}{6}.\end{cases}$$
Introduce the operator 
$$M(h) = P( h) + \Op(\sigma(x', h\xi')). $$
The principal symbol of $M(h)$ has the form
$$m(x', h\xi') =   - \sqrt{1 -h^2 r_0} ( 1- \beta(h^2 r_0))+ \gamma(x) +  \beta(h^2 r_0) \sqrt{1 + h^2 r_0} + \sigma(x', h \xi').$$
The operator  $M(h)$ is elliptic since for $h^2r_0 \leq 1 - \frac{11c_0^2}{12}$ we have $ m \geq \gamma(x) \geq c_0,$ for  $1 - \frac{11 c_0^2}{12} < h^2 r_0 \leq 1 - \frac{ c_0^2}{2}$ we get $\beta(h^2 r_0) = 0$ and
$$m \geq \gamma(x) - \sqrt{1 - h^2r_0} \geq \gamma(x) - \sqrt{\frac{11}{12}}c_0,$$
 while for $ 1- \frac{c_0^2}{2} < h^2 r_0 \leq 1- \frac{c_0^2}{4}$ we obtain $m \geq \gamma(x) - \sqrt{1 - h^2r_0} \geq \gamma(x) - \sqrt{\frac{1}{2}}c_0.$ Finally, for $h^2 r_0 \geq 1 - \frac{c_0^2}{4}$ we have $\beta(h^2 r_0) = 1$ and $m \geq \gamma(x).$

Consequently, $m \in S^{1}_{0, 1}$, the operator $M^{-1} ( h): H^s - H^{s+1}$ is bounded by ${\mathcal O}_s(1)$
and $\widetilde{WF} (P(h) - M(h)) \subset \{(x', \xi'):\:r_0 \leq 1 - \frac{5c_0^2}{6}\} .$  By applying Proposition A.1 
in \cite{SjV}, we can extend homomorphically $\sigma(x', h \xi')$ to $\eta(x', \xi'; \th)$ for $\th \in L$. Thus
$M(h)$ has a holomorphic extension 
$$M(\th) =P(\th) + \Op(\eta(x',\xi'; \th))$$
 for $\th \in L$ and $\widetilde{WF} (P(\th) - M(\th)) \subset \{(x', \xi'): \:r_0  \leq1- \frac{5c_0^2}{6}\} .$ The last relation implies $P(\th) - M(\th) : {\mathcal O}_s(1): H^{-s} \rightarrow H^s, \: \forall s.$

Next, one deduces the estimate
\begin{equation}\label{eq:6.2} 
\|P(\th)^{-1} \|_{\lc(H^s, H^{s+1})} \leq C_s \frac{h}{|\Im \th|}, \: h > 0, \: \Im \th \neq 0
\end{equation}
applying (\ref{eq:6.1}) and the representation
\begin{equation} \label{eq:6.3}
P^{-1} = M^{-1} - M^{-1} (P- M)M^{-1} + M^{-1}(P - M) P^{-1} (P- M) M^{-1},
\end{equation} 
combined with the property of $P(\th) - M(\th)$ mentioned above. 
Following \cite{SjV}, introduce a piecewise smooth simply positively oriented curve $\gamma_{k.p}$ as a union of four segments: $\{ h \in J_{k, p},\: \Im \th = \pm  h^{p+1}\}$ and $\{h \in \pa J_{k, p},\: |\Im \th|\leq h^{p+ 1}\},$ where $J_{k, p}$ is one of the intervals in $\Omega_p$ defined in Proposition 5.2. Then  we have
\begin{prop}[Prop. 5.2, \cite{SjV}] For every $h \in \gamma_{k, p}$ the inverse operator $P(\th)^{-1}$ exists and
$$\|P(\th)^{-1} \|_{\lc(H^s, H^{s+1})} \leq C_s h^{-p},\: \th \in \gamma_{k, p}.$$
\end{prop}

For the operator $P(\th)$, we obtain a trace formula repeating without any change the proof in \cite{SjV}. Let $\mu_k(h_k) = 0, \: k > k_0.$
 Define the multiplicity of $h_k$ as the  multiplicity of the eigenvalue $\mu_k(h_k)$ and denote the derivate of $A$ with respect to $\th$ by $\dot{A}.$ 
\begin{prop} [Prop. 5.3, \cite{SjV}] Let $\tau  \subset L$ be a closed positively oriented $C^1$ curve without self intersections which avoids the points $h_k$ with $\mu_k(h_k) = 0.$ Then 
  $${\rm tr} \frac{1}{2 \pi \ii}  \int_{\tau} P(\th)^{-1} \dot{P}(\th) d\th$$
  is equal to the number of $h_k$ in the domain bounded by $\tau.$   
\end{prop}

Introduce an operator $\chi \in L^{0, 0}_{cl}(\Gamma)$  which is holomorphic for $\th \in L$ so that 
 $$\widetilde{WF}(\chi) \subset \{(x', \xi') \in T^*(\Gamma):\: r_0(x', \xi') \leq 1 - \frac{3c_0^2}{4}\},$$
$$\widetilde{WF}(1-\chi)  \subset \{(x', \xi') \in T^*(\Gamma):\:r_0(x', \xi')  \geq 1-\frac{4c_0^2}{5}\}.$$ 
 We apply (\ref{eq:6.3}) for $P^{-1}(1 - \chi)$, and exploiting
$$\widetilde{WF}(P(\th) - M(\th)) \cap \widetilde{WF}(1 - \chi) = \emptyset,$$
 we get
 \begin{equation} \label{eq:6.4}
P^{-1} = P^{-1} \chi + M^{-1} ( 1- \chi) + K_1,\: \th \in \gamma_{k, p}
\end{equation}
with $K_1: \co(|\th|^s): H^{-s} \to H^s, \: \forall s.$ For the analysis below consider a product $\chi_1 P^{-1} \chi_2,$ where $\chi_1, \chi_2 \in L^{0, 0}_{cl}$ are such that
$\widetilde{WF}(\chi_1) \cap \widetilde{WF} (\chi_2) = \emptyset,$  and  $\widetilde{WF}(\chi_1) $ or $\widetilde{WF}(\chi_2)$ is disjoint from $\{(x', \xi'):\: r_0 \leq 1 - \frac{5 c_0^ 2}{6} \}.$ Applying once more (\ref{eq:6.3}), we deduce $\chi_1 P^{-1}\chi_2 : \co(|h|^s): \: H^{-s}\to H^s,\: \forall s.$
   
       We pass to the analysis of the inverse of the operator $C(\th)$. First choosing small $\omega > 0$, introduce a partition of unity on $T^*(\Gamma)$
       $$1 = \chi^{-}_{\omega} + \chi^0_{\omega} + \chi^{+}_{\omega},$$
where the functions $\chi^{j}_{\omega}(x',\xi') ,\: j = 0, \pm$ have been introduced in Section 4. We replace $\delta$ by $\omega$ to avoid a confusion with $\delta = \frac{c_0^2}{2}$ fixed above and choose $\omega \ll c_0^2$.  Let $Q^{j}_{\omega} = \Op(\chi^{j}_{\omega}(x', h\xi')),\: j = 0, \pm.$ The estimates (\ref{eq:4.5}) and (\ref{eq:4.11}) can be written as estimates for $C(\th) = -\tilde{\nc}(\th)+ \gamma(x),$ because
$$\frac{\nc(h, z)}{\sqrt{z}} = \frac{- \ii h \pa_{\nu} u\vert_{\Gamma}}{\sqrt{z}}= -\ii \th \pa_{\nu} u \vert_{\Gamma}= \tilde{\nc}(\th).$$
Notice that for the proof of (\ref{eq:4.11})  we must take $\omega \ll c_0^2$ sufficiently small.
From these estimates we deduce
$$\|(1 -Q^{-}_{\omega}) f\|_{H^{1/2}} \leq \|Q^{0}_{\omega} f\|_{H^{1/2}} + \|Q^{+}_{\omega}  f \|_{H^{1/2}} $$
$$\leq B\Bigl(\|C(\th) Q^{0}_{\omega} f\|_{H^{-1/2}}  +\|C(\th) Q^{+}_{\omega} f\|_{H^{-1/2}}\Bigr)$$
with constant $B > 0$ depending of $\omega$ and small $h$. We fix $\omega \leq \frac{c_0^2}{12}$ and obtain
$$\|f\|_{H^{1/2}} \leq B \Bigl( \|C(\th) f\|_{H^{-1/2}} + 2\|C(\th) Q^{+}_{\omega} f\|_{H^{-1/2}} $$
$$+ \|C(\th) Q^{-}_{\omega} f\|_{H^{-1/2}}\Bigr) + \|Q^{-}_{\omega} f\|_{H^{1/2}}.$$
By using a parametrix $S_N(h)$ in the elliptic region, we have $C(\th)Q^{+}_{\omega} = S_N(h)Q^{+}_{\omega}  + {\mathcal R_N}, \: {\mathcal R_N} = \co(|\th|^N)$. Then
$$C(\th) Q^{+}_{\omega}  = Q^{+}_{\omega} C(\th) + [S_N(h), Q^{+}_{\omega}] + \co(|\th|^N)$$ 
and it easy to estimate the norm of the commutator $[S_N(h), Q^{+}_{\omega} ]$ by $\co(|\th|).$ Also by using the parametrix $\tilde{T}_N(\th)$ in the hyperbolic region, we obtain the same result for the commutator $[\tilde{T}_N(\th), Q^{-}_{\omega}].$ The terms with norms $\co(|\th|)$ can be absorbed by the left hand side of the last inequality and we give
\begin{equation} \label{eq:6.5}
\|f\|_{H^{1/2}} \leq B_1 \Bigl( \|C(\th) f\|_{H^{-1/2}} + \|Q^{-}_{\omega} f\|_{H^{1/2}}\Bigr).
\end{equation} 
 
Introduce the operator  $\tilde{C}(\th): = C(\th) +  \tilde{\chi}$, where  $\tilde{\chi} \in L^{0,0}_{cl}$ depends homomorphically of $\th \in L$ and $\widetilde{WF}(\tilde{\chi})$ is included in  a neighborhood  of $\Sigma.$ To do this, choose $\zeta \in C_0^{\infty}(T^*(\Gamma); [0, 1])$, so that
$$\zeta(x', \xi') = \begin{cases} 1,\:{\rm if}\: r_0(x', \xi') \leq 1 - \frac{6 c_0^2}{7},\\
0,\: {\rm if}\: r_0(x', \xi')  \geq 1 - \frac{5 c_0^2}{6}.\end{cases}$$
Extend $\zeta$ homomorphically  to $\tilde{\zeta}$ for $\th \in L$ as we have proceeded above with $\eta$, and define the operator $\tilde{\chi}$ with symbol $\tilde{\zeta}.$ Obviously,
$$\widetilde{WF}(\tilde{\chi}) \subset \big\{(x', \xi') \in T^*(\Gamma):\: r_0 \leq 1 - \frac{5c_0^2}{6}\big\},$$
Moreover, it is easy to see that $\tilde{C}(\th)$ is elliptic operator.
Indeed, choose a symbol $\eta_1 \in C_0^{\infty}( T^* (\Gamma); [0, 1])$ satisfying
$$\eta_1(x', \xi') = \begin{cases} 1,\:{\rm if}\: r_0(x', \xi') \leq 1 - \frac{\omega}{2},\\
0,\: {\rm if}\: r_0(x', \xi')  \geq 1 - \frac{\omega}{3}.\end{cases}$$
and set $\upsilon = \Op_h(\eta_1).$
We  apply (\ref{eq:6.5}) for $(1- \upsilon) f$ and  deduce
\begin{equation} \label{eq:6.6}
\|(1 - \upsilon) f\|_{H^{1/2}} \leq B_1\Bigl( \|(1 - \upsilon)C(\th) f\|_{H^{-1/2}} + \|[C(\th) , \upsilon] f\|_{H^{-1/2}} + \co(|\th|^{\infty}) \|f\|_{H^{1/2}}\Bigr).
\end{equation} 
Since $\widetilde{WF}(1 - \upsilon) \cap \widetilde{WF}(\tilde{\chi}) = \emptyset,$ we have $(1- \upsilon) C(\th) = (1- \upsilon) \tilde{C}(\th) + \co(|\th|^{\infty})$  and we may replace
$(1 - \upsilon) C(\th)$ by $(1 - \upsilon)\tilde{C}(\th)$ in the last inequality.
On the other hand,  $\tilde{C}(\th)$ is elliptic in the hyperbolic region and similarly to (\ref{eq:4.5}) we give
\begin{equation} \label{eq:6.7} 
\|\upsilon f\|_{H^{1/2}} \leq B_2 \Bigl(\|\tilde{C}(\th) \upsilon f\|_{H^{-1/2}} + A h \|f\|_{H^{1/2}}\Bigr ).
\end{equation} 
The commutator $[C(\th), \upsilon]$ yields a term with norm $\co(|h|)$ and taking together (\ref{eq:6.6}) and (\ref{eq:6.7}), we conclude that
\begin{equation} \label{eq:6.8} 
\|f\|_{H^{1/2}} \leq B_3 \|\tilde{C}(\th) f\|_{H^{-1/2}},\: \th \in L.
\end{equation} 
This proves that $\tilde{C}(\th)$ is invertible with inverse $\tilde{C}^{-1}(\th)$  holomorphic for $\th \in L$.\\

Passing to $C^{-1}(\th)$, consider the operator
$$D = \tilde{C}^{-1}(\th) (1 - \chi) + P^{-1}(\th) \chi, \: \th \in \gamma_{k, p}.$$
We have
$$C D = I + (C - \tilde{C})\tilde{C}^{-1} ( 1- \chi) + (C - P) P^{-1} \chi.$$
Let $\tilde{\alpha} = \Op(\alpha) \in L^{0, 0}_{cl}$ with symbol $\alpha(x', h\xi')$ be the operator introduced in the preceding section. Clearly,
$$\widetilde{WF}(\tilde{\alpha}) \subset \big\{(x', \xi'):\:r_0 \leq 1 -\frac{c_0^2( 1 + \ep_1)}{2}\big\},\: \widetilde{WF}(1- \tilde{\alpha}) \subset \big\{ (x', \xi'): \:r_0 \geq 1 - \frac{3c_0^2}{5}\big\}.$$
Then $(C - P)(1 - \tilde{\alpha}) P^{-1} \chi: \: \co(|\th|^s): \: H^{-s} \to H^s.$
On the other hand,  according to  (\ref{eq:5.2}), one has
\begin{equation} \label{eq:6.9}
(C- P)\tilde{\alpha}=  {\mathcal R}_N \tilde{\alpha}
\end{equation}
 with ${\mathcal R}_N: \co_s(|\th|^N) : H^s \to H^s$. 
 Applying (\ref{eq:6.2}),we conclude that 
 $$(C - P) P^{-1} \chi = \co_s(|\th|^{N- p}), \: H^{-s} \rightarrow H^s,\:\th \in \gamma_{k, p}.$$
 
Next $C - \tilde{C} = - \tilde{\chi}$ and 
 $\widetilde{WF}(\tilde{\chi}) \cap \widetilde{WF}(1 - \chi) = \emptyset$ implies $\tilde{\chi} \tilde{C}^{-1}(1 - \chi) = \co(|\th|^s): \: H^{-s} \to H^s.$ Finally,
$CD = Id + \co(|\th|^{N- p})$ and we can take $N$ arbitrary large. We obtain an inverse $D( Id + \co(|\th|^{N- p}))^{-1}$ of $C(\th)$  and this implies the following
\begin{prop} For $\th \in \gamma_{k, p}$ we have
\begin{equation} \label{eq:6.10}
\|C^{-1}(\th) \|_{\lc(H^s, H^{s+1})} \leq C_s |\th|^{-p}.
\end{equation} 
\end{prop}
Similarly to (\ref{eq:6.4}), we get
\begin{equation} \label{eq:6.11}
C^{-1}(\th) = \tilde{C}^{-1}(\th) (1- \chi) + P^{-1}(\th) \chi + K_2(\th),\: \th \in \gamma_{k,p}
\end{equation} 
with $K_2(\th): \:\co_s(|\th|^{N - p}): \: H^{-s} \rightarrow H^s.$\\

Now we are going to compare the traces involving $C(\th)$ and $P(\th)$ and our argument is very similar to that in \cite{SjV}.  First, using (\ref{eq:6.11}), we have
$${\rm tr} \frac{1}{2 \pi \ii}\: \int_{\gamma_{k,p}} C^{-1} (\th) \dot{C}(\th) d \th= {\rm tr} \frac{1}{2 \pi \ii}\: \int_{\gamma_{k,p}} P^{-1} (\th) \chi \dot{C}(\th) d \th + \co_p(|\th|^{N - p}).$$
since $\tilde{C}^{-1}(\th) (1- \chi)$ is holomorphic. Next write
$$P^{-1}(\th) \chi \dot{C}(\th) = ( 1-\tilde{\alpha})P^{-1}(\th) \chi \dot{C}(\th) + \tilde{\alpha} P^{-1}(\th) \chi \dot{C}(\th)$$
and observe that $\widetilde{WF}(1- \tilde{\alpha}) \cap \widetilde{WF}(\chi) = \emptyset.$ Thus the norm of first term on the right hand side is estimated by $\co(|\th|^{\infty}).$ For the second one we apply the cyclicity of the trace and obtain
$$P^{-1} (\th ) \chi \Bigl(\dot{C}(\th) - \dot{P}(\th)\Bigr)\tilde{\alpha} + P^{-1}(\th) \chi \dot{P}(\th) \tilde{\alpha}.$$
By using Cauchy formula for the derivative $\frac{\pa}{\pa \th}(C(\th) - P(\th))$ and (\ref{eq:6.9}), we deduce that  $P^{-1}(\th) \chi\Bigl(\dot{C}(\th) - \dot{P}(\th)\Bigr)\tilde{\alpha} = \co(|\th|^{N- p}).$ To handle the term
$P^{-1}(\th) \chi \dot{P}(\th) \tilde{\alpha}$ by the cyclicity of the trace we transfer the operator $\tilde{\alpha}$ on the left and conclude that
$${\rm tr} \frac{1}{2 \pi \ii}\: \int_{\gamma_{k,p}} P^{-1} (\th) \chi \dot{C}(\th) d \th = {\rm tr} \frac{1}{2 \pi \ii}\: \int_{\gamma_{k,p}} P^{-1} (\th) \chi \dot{P}(\th) d \th + \co(|\th|^{N- p}).$$
Finally, taking into account (\ref{eq:6.4}) and the analyticity of $M^{-1} ( 1- \chi)$, we give
$${\rm tr} \frac{1}{2 \pi \ii}\: \int_{\gamma_{k,p}} P^{-1} (\th) \chi \dot{P}(\th) d \th = {\rm tr} \frac{1}{2 \pi \ii}\: \int_{\gamma_{k,p}} P^{-1} (\th) \dot{P}(\th) d \th + \co(|\th|^{\infty}).$$
The difference 
$${\rm tr} \frac{1}{2 \pi \ii}\: \int_{\gamma_{k,p}} C^{-1} (\th) \chi \dot{C}(\th) d \th - {\rm tr} \frac{1}{2 \pi \ii}\: \int_{\gamma_{k,p}} P^{-1} (\th) \chi \dot{P}(\th) d \th$$
 is a negligible term, hence this difference is zero. Repeating the argument in \cite{SjV}, \cite{P2}, we obtain a bijection $(0, h_0] \ni h_k \Leftrightarrow \mu_j \in (\sigma_p(G) \cup {\rm Res}\:(G))$
and we must count the negative eigenvalues $\mu_k(r^{-1})$ of $P(r^{-1}),\: r \geq C_{\gamma}$. By the well known formula (see for instance, Theorem 10.1 in \cite{DS}) we have
\begin{equation} \label{eq:6.12}
\sharp\{k:\:\mu_k(r^{-1}) \leq 0\} = \frac{r^{d-1}}{(2 \pi )^{d-1}}\int_{p_1((x', \xi') \leq 0} dx'd\xi' + \co_{\gamma}(r^{d-2}).
\end{equation} 
Applying (\ref{eq:5.7}), the integration is over $\{(x', \xi'):\:r_0(x', \xi') \leq 1 - \gamma^2(x')\}$  and we obtain the leading term in (\ref{eq:1.6}). This completes the proof of Theorem 1.2.

\appendix
\renewcommand{\theprop}{A.\arabic{prop}}
\renewcommand{\therem}{A.\arabic{rem}}  
\renewcommand{\theequation}{\arabic{section}.\arabic{equation}}
\section*{ Appendix }  \renewcommand{\theequation}{A.\arabic{equation}}
\setcounter{equation}{0}

In this Appendix we assume that  $K = B_3 = \{x \in \R^3: |x| \leq 1\}$  and  $ \gamma \geq 0$ is a constant. We will  prove the following
\begin{prop}For $K = B_3$ and $0 < \gamma < 1$ the operator $G$ has no eigenvalues.
\end{prop} 
\begin{proof}
Consider the dissipative problem
\begin{equation} \label{eq:A.1}
\begin{cases}  (-\Delta + \lambda^2)u =0 \: {\rm in}\: |x| > 1,\\
\pa_r u- \lambda \gamma u = 0,\: {\rm on}\: |x| = 1,\\
u- \lambda-{\rm incoming}.
\end{cases}
\end{equation}
Set in (\ref{eq:A.1}) $\lambda = \ii \mu$. The incoming solution of (\ref{eq:A.1}) in polar coordinates $(r, \omega) \in \R^+ \times \Ss^2$ has the form
$$u(r, \omega, \mu) = \sum_{n = 0}^{\infty} \sum_{m = - n}^{n} a_{n, m} \frac{h^{(1)}_n(\mu r)}{h^{(1)}_n(\mu)} Y_{n, m}(\omega).$$
Here  $h^{(1)}_n(r) = \frac{H^{(1)}_{n+ 1/2}(r)}{r^{1/2}}$ are the spherical (modified) Hankel functions of first kind, 
$Y_{n, m}(\omega)$ are the eigenfunctions of the Laplace-Beltrami operator $- \Delta_{\Ss^2}$ with eigenvalues $n(n+1)$ and
$$u\vert_{|x| = 1} = f(\omega) = \sum_{n = 0}^{\infty} \sum_{m = -n}^{n} a_{n, m} Y_{n, m}(\omega).$$
To satisfy the boundary condition we must have
$$C(n; \mu, \gamma) a_{n, m} = 0,\: \forall n, \: -n \leq m \leq n,$$
where
$$C(n; \mu, \gamma) =\pa_r \Bigl(\frac{h^{(1)}_n(\mu r)}{h^{(1)}_n(\mu)}\Bigr)\Big\vert_{r = 1} - \ii \mu \gamma, \: n \in \N, \Im \mu > 0.$$

Our purpose is to prove that $C(n; \mu, \gamma) \neq 0$ for $\Im \mu > 0$ and all $n \in \N$. This implies $a_{n, m} = 0$ and $f \equiv 0.$ Hence $u \equiv 0$ since the Dirichlet problem has no eigenvalues.
We have the representation
$$h_n^{(1)}(x) = (-1)^{n+ 1} \frac{e^{\ii x}}{x} R_n\Bigl(\frac{\ii}{2x}\Bigr)$$
with
$$R_n(z) := \sum_{m = 0}^n \frac{(n + m) !}{(n - m)!}\frac{z^m}{m !} = \sum_{m= 0}^n c_m z^m.$$
It is well known that the zeros of $H^{(1)}_{n + 1/2}(z)$ are located in the domain $\Im z < 0$, so $R_n(z) \neq 0$ for  $\Re z \geq 0.$
Taking the derivative with respect to $r$, we obtain
$$C(n; \mu, \gamma) = (1 - \gamma) \ii \mu - \sum_{m = 0}^n \frac{(m+ 1) ( n+ m)!}{m ! (n - n)!} \Bigl(\frac{\ii} {2\mu}\Bigr)^m \Bigl(R_n (\frac{\ii}{2 \mu})\Bigr)^{-1}.$$
Setting $w = \frac{\ii}{2 \mu},$ we will study for $\Re w > 0$ the equation
\begin{align} \label{eq:A.2}
-C\Bigl(n; \frac{\ii}{2 w}, \gamma\Bigr)= \frac{1 - \gamma}{2w} +   \sum_{m = 0}^n (m+ 1) c_m w^{m}(R_n(w))^{-1}  \nonumber \\
= \frac{(1 - \gamma) R_n(w) + \sum_{m = 0}^{n} 2 (m + 1)c_m(\gamma) w^{m + 1}}{2 w R_n(w)} =  \frac{\sum_{m= 0}^{n+1}b_m(\gamma) w^m} {2 w R_n(w)} = 0.
\end{align}
The coefficients $b_m(\gamma)$ have the form
$$b_m(\gamma) = (1- \gamma) c_m + 2 m c_{m-1},\: m = 1,..., n,$$
$$b_{n+1}(\gamma) = 2(n+1)c_n = 2(n+1) \frac{(2n)!}{n!},\: b_0(\gamma) = 1- \gamma.$$
We fix $0 < \gamma < 1$ and for $0\leq \ep \leq 1$ consider the polynomial 
$$B_n(w; \ep) : = \sum_{k= 0}^{n+1} b_k(\ep\gamma) w^k.$$
Let $w_j(n; \ep) = 1,...,n+1$ be the roots of the equation $B_n(w; \ep) = 0$ with respect to $w$.
 By using the result of Fujiwara (see \cite{Mar}), we obtain an upper bound for $|w_j(n; \ep)|,\: j = 0,..., n+1$ given by
$$2 \max \Big\{ \Big | \frac{b_n(\ep \gamma)}{b_{n+1}}\Big |, \Big |\frac{b_{n-1}(\ep\gamma)}{b_{n+1}}\Big |^{1/2},...,\Big | \frac{ b_1(\ep \gamma)}{b_{n+1}} \Big |^{1/n}, \Big|\frac{b_0(\ep\gamma)}{2b_{n+1}}\Big |^{1/(n+1)}\Big\} = M_n(\ep).$$

Since $b_{n+1}$ is independent of $\gamma$ and $\ep$ and $0 < b_k(\ep\gamma) \leq b_k(0), \: 0 \leq \ep \leq 1,\: k = 0,...,n$, we deduce 
$$M_n(\ep) \leq  M_n(0),\: 0 \leq \ep \leq 1.$$
 To estimate $M_n(0)$, observe that the sequence $b_{n+1}(\ep\gamma) > b_n(\ep \gamma) > ...> b_1(\ep \gamma) > 1 -\ep\gamma$ is decreasing. Indeed,
 $$c_k - c_{k-1} = \frac{(n + k -1)!}{(k-1)! (n- k)!} \Bigl[ \frac{n+ k}{k} - \frac{1}{n - k + 1}\Bigr]= \frac{(n + k -1)!}{k! (n- k + 1)!} (n^2 - k^2 + n).$$
  Thus the maximal term in the above upper bound becomes
$$ \frac{2b_n(\ep \gamma)}{b_{n+1}} = \frac{( 1- \ep\gamma)c_n + 2n c_{n-1}}{(n+1)c_n}= \frac{1- \ep\gamma}{n +1} + \frac{n}{n+1} = 1 - \frac{\ep\gamma}{n+1}\leq 1$$
and we take $M_n(0) = 1.$

Consider the contour $\omega = \alpha \cup \beta \subset \{w\in \C:\:\Re w \geq 0\}$, where
 $$\alpha = \{w \in \C: |w| = 2, \Re w \geq 0\},\:\beta = \{ w = \ii y \in \C: |y| \leq 2\}.$$
By using the factorisation of $B_n(w; \ep),$ we give
\begin{equation} \label{eq:A.3}
|B_n(w; \ep)| \geq b_{n+1},\: \forall w \in \alpha,\:0 \leq \ep \leq 1.
\end{equation}
On the other hand, $B_n(w; \ep) \neq 0$ for $w \in \beta,\: 0 \leq \ep \leq 1.$ Indeed, $B_n(0; \ep) = 1- \ep \gamma \neq 0$ and for $w = \ii y, y \neq 0$ one has $C(n; \frac{1}{2 y}, \ep\gamma)  \neq 0$. In fact, $\R \ni \frac{1}{2 y} =- \ii \lambda$ yields $\lambda \in \ii \R \setminus \{0\}$ and in Section 2 it was shown that there are no eigenvalues and incoming resonances on $\ii \R \setminus \{0\}.$  If $C(n; -\ii \lambda_0, \ep\gamma) =  0$ for some $\lambda_0 \in \R$ and some $n$, taking $a_{n, m} \neq 0$ and $a_{k, m} = 0, \: k \neq n,$ we obtain a function $f \neq 0$ such that $\cc(-\ii \lambda_0) f = 0$ which is impossible. We claim that there exists $\delta_0(n) > 0$ such that
\begin{equation}\label{eq:A.4}
|B_n(w; \ep)| \geq \delta_0(n),\: \forall w \in \beta,\: 0 \leq \ep \leq 1.
\end{equation}
 Assume (\ref{eq:A.4}) not true. Then there exists a sequence $\{\ii y_m,\ep_m\}\in \beta \times [0, 1]$ such that
$$|B_n(\ii y_m; \ep_m)|\leq \frac{1}{m}, \: \forall m \in \N.$$
Choosing convergent subsequence $\{y_{m_k}, \ep_{m_k}\}$ and passing to limit $(y_{m_k}, \ep_{m_k}) \to (y_0, \ep_0) \in \beta \times [0, 1]$, we obtain a contradiction with $B_n(\ii y_0, \ep_0) \neq 0.$ This proves the claim.

Now consider the integral
$$q_n(\ep) = \frac{1}{2 \pi \ii} \int_{\omega} \frac{ B_n'(w; \ep)}{B_n(w; \ep)} dw,\: 0 \leq \ep \leq 1.$$
Here $q_n(\ep) \in \N$ is equal to the number of the roots of $B_n(w; \ep) = 0$ counted with their multiplicities lying in the interior of the domain bounded by $\omega$. We will prove that $q_n(\ep)$ depends continuously of $\ep \in [0, 1],$ hence $q_n(\ep)$ is constant.
Let $\ep_1, \ep_2 \in [0, 1].$ Write
$$\frac{ B_n'(w;\ep_1)}{B_n(w;\ep_1)} - \frac{ B_n'(w; \ep_2)}{B_n(w; \ep_2)}$$
$$=B_n(w; \ep_1)^{-1} B_n(w; \ep_2)^{-1} \Bigl([B_n'(w; \ep_1) - B_n'(w; \ep_2)]B_n(w; \ep_2)$$
$$+ B_n'(w; \ep_2)[B_n(w; \ep_2) - B_n(w; \ep_1)]\Bigr).$$
On the other hand,
$$B_n(w; \ep_1) - B_n(w; \ep_2)=  (\ep_2- \ep_1)\gamma \sum_{m = 0}^n c_m w^m$$
and a similar equality holds for $B_n'(w; \ep_1) - B_n'(w; \ep_2).$ Taking into account (\ref{eq:A.3}) and (\ref{eq:A.4}), we have an upper bound
\begin{equation}\label{eq:A.5}
|B_n^{-1}(w, \ep) | \leq \max\big\{ \frac{1}{b_{n+1}}, \frac{1}{\delta_0(n)}\big\},\: \forall w \in \omega,\: 0 \leq \ep \leq 1
\end{equation}
 and we conclude that $q_n(\ep)$ is continuous. This implies $q_n(\ep) = q_n(0) = 0$ since for Neumann problem ($\gamma = 0$) we have no roots of the equation $B_n(w; 0)= 0$ with $\Re w > 0.$ 
\end{proof}
\begin{rem}
 A shorter proof of Proposition $A.4$ may be given by using the continuity of the roots $w_j(n; \ep)$ with respect to $\ep.$ If we have a root $w_j(n; \ep)$ with $\Re w_j(n; \ep) > 0$ for some $0 < \ep < 1,$ then for $\ep \searrow 0$ we obtain $\re w_j(n;0) \geq 0, \: w_j(n; 0) \neq 0$ since the roots $w_j(n;\ep)$ cannot cross the imaginary axis. This leads to a contradiction with the fact that $B_n(w; 0) \neq 0$ in $\{\Re w \geq 0\}.$ The above proof is based on complex analysis and the same approach could be useful for the general case of strictly convex obstacles.
\end{rem}

We may apply another perturbation argument. Consider the polynomial 
$$F_n(w; \eta) = \sum_{m = 0}^{n+1} b_m(1 + \eta) w^m$$  for $|\eta| \ll 1$. Let $w_j(\eta)$ be the roots of $F_n(w; \eta) = 0$ with respect to $w.$ Clearly, for $|\eta|$ small enough we have a simple root $w(\eta)$ such that $w(0) = 0$ and  $\frac{\pa F_n}{\pa w}(0; 0) = 2.$ The other roots of $F_n(w; 0) = 0$ are different from 0 and they have  strictly negative real part (see Appendix in \cite{P1}).  The root $w(\eta)$ must be real, otherwise we will have two perturbed roots for $\eta$ close to 0. Taking the derivative with respect to $\eta$, we obtain
$$\frac{\pa F_n}{\pa w}(w(\eta); \eta) w'(\eta) + \frac{\pa F_n}{\pa \eta}(w(\eta); \eta) = 0$$
and for $\eta = 0$ we deduce $w'(0) = 1/2.$ This implies $\pm w(\eta) > 0$ for $\pm \eta > 0$ and small $|\eta|$ ($w(\eta) < 0$ for $\eta < 0$ follows also from the fact that $B_n(w; \gamma) = 0$ with $0 < \gamma < 1$ has no positive real roots).
For  $\gamma = 1 + \eta > 1$ the root $w(\eta)$ yields an eigenvalue $\lambda = - \frac{1}{2 w(\eta)} < 0$, while for $\gamma = 1 - \eta < 1$ we have an incoming resonance $\lambda = - \frac{1}{2 w(-\eta)} > 0.$ The other roots $w_j(\eta)$ with $\Re w_j(0) < 0$  remain in the half plane $\{\Re z < 0\}$ for small $|\eta|.$

\vspace{0,5cm}

  \end{document}